\documentclass[x11names]{amsart}

\usepackage{amssymb,amsmath,amsthm,amsfonts,bbm,accents,mathtools,dsfont}
\usepackage{enumerate}
\usepackage[margin=0.95in]{geometry}
\usepackage{graphicx,booktabs}
\usepackage{setspace}
\usepackage{textcomp}
\usepackage{stmaryrd}
\usepackage{centernot}
\usepackage{mathdots}
\usepackage{pgf,tikz}
\usepackage{ctable}
\usepackage{mathrsfs}
\usetikzlibrary{arrows.meta}
\usetikzlibrary{positioning}
\usepackage{amsthm}
\usepackage{cancel}
\usetikzlibrary{cd}
\usepackage{caption}

\begingroup
    \makeatletter
    \@for\theoremstyle:=definition,remark,plain\do{%
        \expandafter\g@addto@macro\csname th@\theoremstyle\endcsname{%
            \addtolength\thm@preskip\parskip
            }%
        }
\endgroup

\DeclareMathOperator{\Aut}{Aut}

\DeclareMathOperator{\vspan}{span}

\DeclareMathOperator{\cl}{cl}

\DeclareMathOperator{\St}{St}

\DeclareMathOperator{\Susp}{Susp}

\DeclarePairedDelimiter{\abs}{\lvert}{\rvert}
\DeclarePairedDelimiter{\norm}{\lVert}{\rVert}
\DeclarePairedDelimiter{\group}{\langle}{\rangle}
\DeclarePairedDelimiter{\set}{\{}{\}}
\DeclarePairedDelimiter{\term}{(}{)}

\DeclarePairedDelimiter{\floor}{\lfloor}{\rfloor}

\newcommand{\R}{\mathbf{R}}
\newcommand{\N}{\mathbf{N}}
\newcommand{\C}{\mathbf{C}}
\newcommand{\Z}{\mathbf{Z}}
\newcommand{\Q}{\mathbf{Q}}

\newcommand{\T}{\mathbf{T}}

\newcommand{\Mean}{\mathcal{M}}
\newcommand{\Subs}{\mathcal{S}}
\newcommand{\trop}{\mathrm{trop}}

\newcommand{\marg}[1]{\null}
\newcommand{\edit}[1]{\null}

\newtheorem{thm}{Theorem}[section]
\newtheorem{innercustomthm}{Theorem}
\newenvironment{numthm}[1]
  {\renewcommand\theinnercustomthm{#1}\innercustomthm}
  {\endinnercustomthm}

\newtheorem*{thm*}{Theorem}
\newtheorem{lemma}[thm]{Lemma}
\newtheorem{prop}[thm]{Proposition}

\newtheorem{rmk}[thm]{Remark}
\newtheorem{cor}[thm]{Corollary}
\newtheorem{question}[thm]{Question}

\newtheorem{defn}[thm]{Definition}

\renewcommand{\le}{\leqslant}
\renewcommand{\ge}{\geqslant}

\usepackage{hyperref}
\hypersetup{
  colorlinks   = true,          
  urlcolor     = Chartreuse4,          
  linkcolor    = Firebrick3,          
  citecolor   = DodgerBlue3             
}

\title{Topology of Tropical Moduli of Weighted Stable Curves}

\author[]{Alois Cerbu, Steffen Marcus, Luke Peilen, Dhruv Ranganathan, Andrew Salmon}

\begin{document}
\begin{abstract}
The moduli space $\Delta_{g,w}$ of tropical $w$-weighted stable curves of volume $1$ is naturally identified with the dual complex of the divisor of singular curves in Hassett's spaces of $w$-weighted stable curves. If at least two of the weights are $1$, we prove that $\Delta_{0,w}$ is homotopic to a wedge sum of spheres, possibly of varying dimensions. Under additional natural hypotheses on the weight vector, we establish explicit formulas for the Betti numbers of the spaces. We exhibit infinite families of weights for which the space $\Delta_{0,w}$ is  disconnected and for which the fundamental group of $\Delta_{0,w}$ has torsion. In the latter case, the universal cover is shown to have a natural modular interpretation. This places the weighted variant of the space in stark contrast to the heavy/light cases studied previously by Vogtmann and Cavalieri-Hampe-Markwig-Ranganathan. Finally, we prove a structural result relating the spaces of weighted stable curves in genus $0$ and $1$, and leverage this to extend several of our genus $0$ results to the spaces $\Delta_{1,w}$.
\end{abstract}

\maketitle

\section{Introduction}

The moduli space $\Delta_{g,w}$ is a topological space that parameterizes $n$-marked, $w$-weighted stable tropical curves of genus $g$ with volume (sum of all edge lengths)\edit{P1Q2: added early definition of volume. The next sentence we believe adequately motivates this restriction.} equal to $1$. It has a natural interpretation as the dual complex of the divisor of singular curves in Hassett's moduli space $\overline{\mathcal M}_{g,w}$ of weighted stable curves~\cite{cavalieri2016moduli,Has03,Ulirsch15}. In this paper, we investigate the homotopy type of $\Delta_{g,w}$ in genus $0$ and $1$ as the weights vary.

\begin{numthm}{A}\label{A} Let $w = (1,1,w_3,\dots,w_{n})$. The moduli space $\Delta_{0,w}$ is homotopic to a wedge sum of spheres, possibly of varying dimensions. In particular, if $0 < \epsilon \le 1/k$ and $w = (1^{(m)},\epsilon^{(k)})$, then $\Delta_{0,w}$ is homotopic to a wedge sum of $(m-2)!\,(m-1)^k$ spheres of dimension $m+k-4$. If we fix $m = 2$ and allow $\epsilon = 1/\ell$ for any $\ell$, there are explicit closed formulas for the Betti numbers of $\Delta_{0,w}$.
\end{numthm}

If the supposition on the weight vector in Theorem~\ref{A} is dropped, the conclusion can fail dramatically: if $w$ does not have two weight $1$ entries, $\Delta_{0,w}$ may be disconnected or $\pi_1(\Delta_{0,w})$ may have torsion.

\begin{numthm}{B}\label{B}
There exist infinite families of weight vectors $w$ such that $\Delta_{0,w}$ is a disjoint union of a wedge sum of spheres, and such that $\Delta_{0,w}$ has fundamental group $\Z/2$. Specifically, \begin{enumerate}[\bf I.]
\item Suppose $\epsilon \le 1/k$. If $w = \big(\frac{1}{m}^{(2m)}, \epsilon^{(k)}\big)$ for $k \ge 2$,
then $\Delta_{0,w}$ is homeomorphic to a disjoint union of $\frac12 \binom{2m}{m}$ spheres of dimension $k-2$.
\item If $w = \big(\frac{1}{k}^{(2k+2+m)}\big)$ for $2 \le m \le k$, then $\pi_1(\Delta_{0,w}) = \Z / 2$ and the universal cover is homotopic to a wedge sum of spheres of dimension $m$.
\end{enumerate}
\end{numthm}

In Theorem~\ref{genus one action} we establish a structural relationship between the spaces of tropical weighted stable curves in genus $0$ and $1$. We leverage Theorem~\ref{A} to deduce results about the structure of $\Delta_{1,w}$. A further discussion of genus $1$ appears in Section~\ref{sec: genus1}.

\begin{numthm}{C}\label{C}
Let $w = (1, 1, w_3, \dots, w_n)$ where $n \ge 3$. The moduli space $\Delta_{1,w}$ is homotopic to a wedge sum of spheres, possibly of varying dimensions. Furthermore, if $w = (1^{(m)}, \epsilon^{(k)})$ for $m \ge 2$, $m + k \ge 3$, and $0<\epsilon\le 1/k$,
then $\Delta_{1,w}$ is homotopic to a wedge sum of $\frac{1}{2}(m-1)!\, m^k$ spheres of dimension $m+k-1$.
\end{numthm}

We supplement this analysis with a number of explicit calculations in {\tt \href{http://www.sagemath.org}{SageMath}} \cite{git,sage}. To convey the range of behavior in the topology, several low dimensional genus $0$ examples are recorded in Section~\ref{sec:examples}. 

\subsection{Motivation} The geometry of the moduli spaces $\overline{\mathcal M}_{g,w}$ motivates the study of $\Delta_{g,w}$. Such spaces were introduced by Hassett as alternate compactifications of the moduli space $\Mean_{g,n}$ of smooth pointed genus $g$ curves~\cite{Has03} and continue to be heavily studied. For instance, the cohomology of the weighted spaces informs the fine structure of the cohomology of $\overline{\Mean}_{0,n}$, see~\cite{BM13,BM14}. When $w = (1^{(m)},\epsilon^{(k)})$ and $\epsilon = 1/\ell$, the spaces $\Delta_{0,w}$ appear prominently in ongoing work of Castravet and Tevelev on the derived category of $\overline\Mean_{0,n}$.

The compact moduli space $\overline\Mean_{g,w}$ contains an open subset $\Mean_{g,w}$ -- in general strictly containing\edit{P2Q1: Edited for clarity here.} $\Mean_{g,n}$ -- consisting of weighted stable curves $(C,p_1,\ldots, p_n)$ such that $C$ is smooth. This furnishes a natural sequence of open immersions
\[
\Mean_{g,n}\subset \Mean_{g,w} \subset \overline{\Mean}_{g,w}.
\]
When the weights are all equal to $1$, it follows from a result of Abramovich, Caporaso, and Payne that the dual complex of the normal crossings divisor $\overline{\mathcal M}_{g,n}\setminus \mathcal M_{g,n}$ is naturally identified with $\Delta_{g,n}$, the tropical moduli space~\cite{ACP15}. By well-known results concerning boundary complexes and weight filtrations, the reduced rational homology of the dual complex $\Delta(\mathcal D)$ of a normal crossings pair $(\mathcal X,\mathcal D)$ computes the top graded piece of the weight filtration on the cohomology of the open stack $\mathcal X\setminus \mathcal D$, see~\cite{CGP,hackingmodulispaces,sampaynemichigan}, and in particular \cite[Proposition 5.6]{CGP}. While for each pair $(g,n)$ this computation is in principle a finite combinatorial topology problem, the complexity grows rapidly. For instance, at the time of writing, it is unknown whether $\Delta_g$ is simply connected for $g\ge 3$.

For general \edit{P2Q2-2:Changed "weights" to "weight vector $w$", so as to not confuse with the use of this word in ``...top weight piece of cohomology..." below} weight vector $w$, the divisor $\overline{\Mean}_{g,w}\setminus \Mean_{g,n}$ does not have normal crossings. \footnote{Consider, for example, the case $w = (1,1,\epsilon,\epsilon,\epsilon)$ in which the moduli space has dimension two; the divisor $\overline{\mathcal M_{0,w}} \setminus \mathcal M_{0,n}$ is not normal crossing as it includes three curves meeting at a point.} \edit{P2Q2-1: Provided an example in footnote for the reader.}However, the smaller complement $\overline{\Mean}_{g,w}\setminus \Mean_{g,w}$ does have normal crossings, and this motivates our study of the open stacks $\Mean_{g,w}$. The dual complex of this normal crossings divisor is naturally identified with the space of $w$-weighted stable tropical curves of genus $g$, see~\cite{cavalieri2016moduli,Ulirsch15}. From the perspective of logarithmic geometry, the divisorial pair $(\overline{\Mean}_{g,w},\overline{\Mean}_{g,w}\setminus \Mean_{g,w})$ is logarithmically smooth, and the universal weighted stable curve is a logarithmically smooth fibration.

Roughly speaking, as the entries of $w$ decrease, the stack $\mathcal M_{g,w}$ becomes larger and closer to projective. As smooth projective varieties have trivial top graded piece in the weight filtration, it is natural to expect that the complexity of $\Delta_{g,w}$ decreases as the entries of $w$ decrease. As the weight vector changes, one obtains a network of moduli spaces $\mathcal M_{g,w}$ with open immersions between them. In genus $0$ and $1$, our results demonstrate how the top weight piece of the cohomology of these open varieties changes as $w$ varies.

\subsection{Related work} Our results are related to, contrast with, and rely on the work of many others. In~\cite{vogtmannouterspace} Vogtmann established that $\Delta_{0,n}$ is homotopic to a wedge sum of $(n-2)!$ spheres of top dimension using the topology of partially ordered sets. Robinson and Whitehouse rederived the homotopy type of $\Delta_{0,n}$ by identifying a large contractible subcomplex $X_{0,n}$ of $\Delta_{0,n}$ for which $\Delta_{0,n}/X_{0,n}$ is manifestly a wedge sum of spheres. This allowed them to determine the $S_n$ representation on the cohomology of $\Delta_{0,n}$ obtained by permuting the marked points~\cite{robinson1996tree}.

Some results are also known in the weighted case. When $w = (1^{(m)},\epsilon^{(k)})$ for $\epsilon$ sufficiently small, Cavalieri, Hampe, Markwig, and Ranganathan identified $\Delta_{0,w}$ with the link at the origin of the Bergman complex of a certain graphic matroid. \edit{P2Q3: added citation to ardila2006.} Along with results from \cite{ardila2006bergman}, it follows that $\Delta_{0,w}$ is homotopic to a wedge of spheres of top dimension~\cite{cavalieri2016moduli}. In Section~\ref{sec: heavy-light}, we give a direct proof of this fact that does not rely on the structure of $\Delta_{0,w}$ as a Bergman complex, and allows us to give a simple formula for the number of spheres.

Work in the higher genus case is much more recent. The homotopy type of $\Delta_{1,n}$ was determined to be a wedge sum of $\frac{1}{2}(n-1)!$ spheres of top dimension by Chan, Galatius, and Payne~\cite[Section 9]{CGP}. This has led to several applications to the topology of the moduli spaces $\Mean_{1,n}$. These authors also show that $\Delta_{g,n}$ is at least $(n-3)$-connected, and that $\Delta_{1,n}$ is $(n-2)$-connected. In contrast, $\Delta_{0,w}$ may be disconnected even if the number of marked points is large.

The reduced rational homology of $\Delta_{2,n}$ was calculated for $n\le 8$ by Chan~\cite{Chan15}. She also showed that the homology of the spaces $\Delta_{2,n}$ can have torsion. However, the torsion homology groups of $\Delta_{2,n}$ appear only in high degree, so the presence of torsion in the fundamental group of $\Delta_{0,w}$ is a new phenomenon for tropical moduli spaces. Nonetheless, with additional hypotheses on $w$, Theorems~\ref{A} and~\ref{C} show that $\Delta_{0,w}$ and $\Delta_{1,w}$ exhibit good topological properties -- they are homotopic to wedges of spheres of varying dimensions.

Note that the genus $1$ result $\Delta_{1,n} \simeq \bigvee_{\frac{1}{2}(n-1)!} S^{n-1}$ of \cite[Section 9]{CGP} bears a cosmetic resemblance to the genus $0$ result $\Delta_{0,n} \simeq \bigvee_{(n-2)!} S^{n-4}$ of \cite{vogtmannouterspace}. In Theorem~\ref{genus one action}, we prove a structural result that explains this similarity. We leverage this relationship to prove Theorem~\ref{C}.

The methods we employ involve subspace arrangements, naturally generalizing those of~\cite{cavalieri2016moduli,vogtmannouterspace}, as well as the generalized discrete Morse theory ideas of~\cite{Chan15,CGP}. Namely, we identify a contractible subcomplex $X_{0,w}$ of $\Delta_{0,w}$ and understand $\Delta_{0,w}/X_{0,w}$ in terms of the homotopy type of a diagonal subspace arrangement. Such subspace arrangements have been heavily studied~\cite{bjornerwachsnonpureshellability,goresky1988stratified,kimshellablelattice,ziegler1993intersectionlattice}. We combine the modular interpretation of $\Delta_{0,w}/X_{0,w}$ with a shelling argument to establish Theorem~\ref{A}. Explicit formulas for the homology in special cases follow from prior results on the corresponding subspace arrangements by~\cite{bjorner1995homology}.

\subsection*{Acknowledgements} This project was completed as part of the 2017 Summer Undergraduate Mathematics Research at Yale (S.U.M.R.Y.) program. We are grateful to all the participants for helping to create a stimulating mathematical environment. The research presented here benefited from conversations with Kenny Ascher, Dori Bejleri, Melody Chan, Netanel Friedenberg, Dave Jensen, Sam Payne, and Jonathan Wise, and originates from discussions with Renzo Cavalieri, Simon Hampe, and Hannah Markwig. D.R. was partially supported by NSF grant number DMS-1128155 (Institute for Advanced Study). The final draft was improved by the careful reading of an anonymous referee. 

\section{Tropical moduli spaces of curves}

We briefly recall the construction of the moduli space $\mathcal M_{g,w}^{\trop}$ and its link, referring the reader to~\cite{ACP15,CGP,Ulirsch15} for additional details.

Fix integers $n$ and $g$ and a \textbf{weight vector} $w = (w_1,w_2,\dots, w_n) \in (0,1]^{n}$ such that
\[
2g-2+\sum_{i=1}^n w_i>0.
\]
An \textbf{$n$-marked graph} is a connected graph $G$ together with a \textbf{genus function} $g: V(G)\to \Z_{\ge 0}$ and a \textbf{marking function} $m_G: \set{1,\dots, n} \to V(G)$. The \textbf{genus} of an $n$-marked graph $G$ is the sum
\[
g(G) = b_1(G)+\sum_{v\in V} g(v),
\]
where $b_1(G)$ is the first Betti number of the graph. Given a vertex $v\in V(G)$, define the \textbf{$w$-weighted valency} of $v$ to be the number of flags of edges incident to $v$, plus the sum of the weights of the marks based at $v$. Finally, a $w$-stable $n$-marked genus $g$ graph is said to be \textbf{stable} if for every vertex $v$ of genus $0$, the $w$-weighted valency is strictly larger than $2$. Examples of stable and unstable marked graphs are provided in figure \ref{fig:tropcurve}.
\edit{P3Q2: Added figure and caption.}

\begin{figure}
\begin{tikzpicture}[line cap = round, line join = round]
\coordinate (core) at (0,0);
\coordinate (rvert) at (1,0);
\coordinate (ulvert) at (-0.5,0.8660254038);
\coordinate (dlvert) at (-0.5,-0.8660254038);
\node (x1) at (-0.2411809549,1.8319512301) {$(1)$};
\node (x4) at (-1.4659258263, 1.1248444489) {$(4)$};

\node (x2) at (-1.4659258263, -1.1248444489) {$(2)$};
\node (x5) at (-0.2411809549,-1.8319512301) {$(5)$};

\node (x3) at (1.5, 0.8660254038) {$(3)$};
\node (x6) at (2, 0) {$(6)$};
\node (x7) at (1.5, -0.8660254038) {$(7)$};

\draw (core) -- (rvert)
	  (core) -- (ulvert)
	  (core) -- (dlvert);

\draw[gray,dashed] (ulvert) -- (x1)
                   (ulvert) -- (x4)
                   (dlvert) -- (x2)
                   (dlvert) -- (x5)
                   (rvert) -- (x3)
                   (rvert) -- (x6)
                   (rvert) -- (x7);
\end{tikzpicture}
\qquad\qquad \begin{tikzpicture}[line cap = round, line join = round]
\coordinate (core) at (0,0);
\coordinate (rvert) at (1,0);
\coordinate (ulvert) at (-0.5,0.8660254038);
\coordinate (dlvert) at (-0.5,-0.8660254038);
\node (x1) at (-0.2411809549,1.8319512301) {$(1)$};
\node (x4) at (-1.4659258263, 1.1248444489) {$(3)$};

\node (x2) at (-1.4659258263, -1.1248444489) {$(2)$};
\node (x5) at (-0.2411809549,-1.8319512301) {$(5)$};

\node (x3) at (1.5, 0.8660254038) {$(4)$};
\node (x6) at (2, 0) {$(6)$};
\node (x7) at (1.5, -0.8660254038) {$(7)$};

\draw (core) -- (rvert)
	  (core) -- (ulvert)
	  (core) -- (dlvert);

\draw[gray,dashed] (ulvert) -- (x1)
                   (ulvert) -- (x4)
                   (dlvert) -- (x2)
                   (dlvert) -- (x5)
                   (rvert) -- (x3)
                   (rvert) -- (x6)
                   (rvert) -- (x7);
\end{tikzpicture}
\caption{For $w = (1,1,1/2,1/3,1/3,1/3,1/4)$, examples of stable (left) and unstable (right) genus-zero $7$-marked graphs. A number $n$ attached to a vertex $v$ by a dotted line indicates that $m_G(n) = v$. The graph on the right is not stable since the rightmost leaf has $w$-weighted valency $1 + w_4 + w_6 + w_7 = 1 + 1/3 + 1/3 + 1/4 = 23/12 \le 2$.}
\label{fig:tropcurve}
\end{figure}
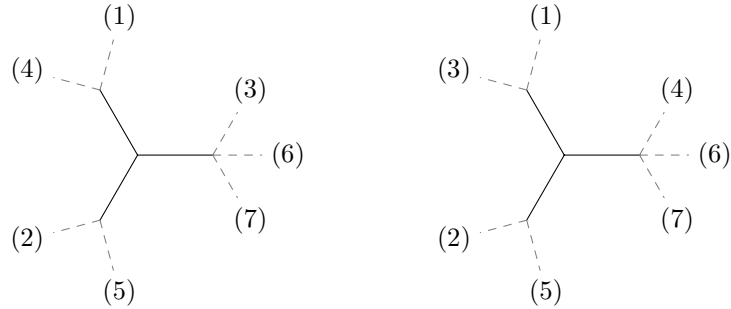

Given a $w$-stable marked graph $G$, a \textbf{$w$-stable tropical curve with underlying graph $G$} is obtained by assigning positive real lengths to the edges of $G$. We restrict throughout to this setting wherein our edge lengths \emph{do not} take infinite lengths. This is a matter of convention; one could alternately allow infinite lengths as in \cite{Ulirsch15} and obtain canonically compactified extended cone complexes. \edit{P3Q3: Identified our choice to not allow infinite edge lengths and added a citation to Ulirsch.}
The space of such $w$-stable tropical curves with a fixed identification of the underlying graphs as $G$ is thus canonically identified with the open cone $\R^{E(G)}_{>0}$. The \textbf{moduli space of $w$-stable tropical curves of type $G$} is
\[
\sigma_G^\circ = \R_{>0}^{E(G)}/  \mathrm{Aut}(G),
\]
where $\mathrm{Aut}(G)$ is the group of automorphisms of the graphs preserving the marking function.

In degenerate cases, when certain edge lengths are equal to $0$, one may perform a series of edge contractions on $G$ to obtain a $w$-stable genus $g$ graph $G'$ with a natural marking and genus function. If $G'$ is such a contraction obtained from $G$, then $\sigma^{\circ}_{G'}$ can be identified with a cell in the space
\[
\sigma_G = \R_{\ge 0}^{E(G)}/\mathrm{Aut}(G).
\]
Ranging over all $G$ and assembling the resulting cells $\sigma_G$, we obtain a topological space parametrizing weighted stable tropical curves of genus $g$
\[
\Mean_{g,w}^{\trop} = \bigsqcup_{G} \term*{\R_{\ge 0}^{\#E(G)} / \Aut(G)} / \sim
\]
where $\sim$ is the identification of cells under graph contractions described above.

Define the \textbf{volume} of a tropical curve to be the sum of the lengths of all edges. This gives rise to a continuous function
\[
\Mean^{\trop}_{g,w}\to \R_{\ge 0}.
\]
The \textbf{link} $\Delta_{g,w}$ of the moduli space $\mathcal M_{g,w}^{\trop}$ is defined to be the locus of tropical curves of volume $1$.

We note that when $g = 0$, the underlying graphs of $w$-stable tropical curves are trees. As there are no marking-preserving automorphisms of such trees, the space $\Mean_{0,w}$ is a cone complex, and the volume $1$ locus $\Delta_{0,w}$ has the structure of a simplicial complex. We will use this structure throughout.

\section{Homology of $\Delta_{0,w}$: computations}\label{sec:examples}

The spaces of weighted stable tropical curves exhibit a range of topological behavior that is not seen in the weight $1$ case considered by Vogtmann and Chan--Galatius--Payne~\cite{vogtmannouterspace,CGP}. We record several examples of weight vectors $w$ for which $\Delta_{0,w}$ exhibits homology in a range of degrees, disconnectedness, and torsion in the fundamental group.

When the weight vector is $w = (1,1,\epsilon,\ldots,\epsilon)$ for $\epsilon \ll 1$, the moduli space $\overline{\mathcal M}_{0,w}$ is known as the \textbf{Losev-Manin} moduli space. This space is isomorphic to the toric variety associated to the permutohedron, and as a consequence $\Delta_{0,w}$ in this case is homeomorphic to a single sphere. As $\epsilon$ increases the topology becomes more subtle. Theorem~\ref{thm:gaps} accounts for the apparent patterns in table \ref{tab:betti}.

\edit{P4Q1Q2Q3: Added labels and captions to tables.}
\begin{center}
\begin{tabular}{c|llllll}
$w$ & $\widetilde H_0$ & $ H_1$ & $ H_2$ & $ H_3$ & $ H_4$ & $ H_5$ \\ \hline
$(1^{(2)},1/2^{(3)})$\rule{0pt}{3ex}  & $\Z$ & $\Z$ \\
$(1^{(2)},1/2^{(4)})$ & $0$ & $\Z^7$ & $\Z$ \\
$(1^{(2)},1/2^{(5)})$ & $0$ & $0$ & $\Z^{31}$ & $\Z$ \\
$(1^{(2)},1/2^{(6)})$ & $0$ & $0$ & $\Z^{20}$ & $\Z^{111}$ & $\Z$ \\
$(1^{(2)},1/2^{(7)})$\rule[-2ex]{0pt}{0pt} & $0$ & $0$ & $0$ & $\Z^{350}$ & $\Z^{351}$ & $\Z$ \\ \hline
$(1^{(2)},1/3^{(4)})$\rule{0pt}{3ex} & $\Z$ & $0$ & $\Z$ \\
$(1^{(2)},1/3^{(5)})$ & $0$ & $\Z^9$ & $0$ & $\Z$ \\
$(1^{(2)},1/3^{(6)})$ & $0$ & $0$ & $\Z^{49}$ & $0$ & $\Z$ \\
$(1^{(2)},1/3^{(7)})$\rule[-2ex]{0pt}{0pt} & $0$ & $0$ & $0$ & $\Z^{209}$ & $0$ & $\Z$ \\ \hline
$(1^{(2)},1/4^{(5)})$\rule{0pt}{3ex} & $\Z$ & $0$ & $0$ & $\Z$ \\
$(1^{(2)},1/4^{(6)})$ & $0$ & $\Z^{11}$ & $0$ & $0$ & $\Z$ \\
$(1^{(2)},1/4^{(7)})$\rule[-2ex]{0pt}{0pt} & $0$ & $0$ & $\Z^{71}$ & $0$ & $0$ & $\Z$ \\ \hline
$(1^{(2)},1/5^{(6)})$\rule{0pt}{3ex} & $\Z$ & $0$ & $0$ & $0$ & $\Z$ \\
$(1^{(2)},1/5^{(7)})$\rule[-2ex]{0pt}{0pt} & $0$ & $\Z^{13}$ & $0$ & $0$ & $0$ & $\Z$
\end{tabular}
\captionof{table}{Reduced homology of $\Delta_{0,w}$ for $w = (1,1,1/\ell,\dots,1/\ell)$} \label{tab:betti}
\end{center}

If the condition that all non-heavy weights are the same is dropped, the space $\Delta_{0,w}$ may have homology in a wide range of degrees, as in table \ref{tab:widerange}.

\begin{center}
\begin{tabular}{c|llllll}
$w$ & $\widetilde H_0$ & $ H_1$ & $ H_2$ & $ H_3$ & $ H_4$ & $ H_5$ \\ \hline
$(1^{(2)},1/2,1/3,1/4,1/5,1/6,1/7,1/7)$\rule{0pt}{3ex} & $0$ & $0$ & $\Z^{14}$ & $\Z^{58}$ & $\Z^3$ & $\Z$\\
$(1^{(2)},1/3,1/4,1/5,1/6,1/7,1/8,1/9)$\rule[-2ex]{0pt}{0pt} & $0$ & $\Z^2$ & $\Z^{23}$ & $0$ & $0$ & $\Z$
\end{tabular}
\captionof{table}{$\Delta_{0,w}$ may have homology supported in a range of degrees} \label{tab:widerange} 
\end{center}

Note that the number and dimension of the spheres given by Theorem \ref{A} varies widely with the type of weight vector. Furthermore, relaxing the requirement that $w$ has two marks of weight one produces examples of disconnectedness and torsion in $\pi_1(\Delta_{0,w})$, as in table \ref{tab:pathological}.

\begin{center}
\begin{tabular}{c|llllll}
$w$ & $\widetilde H_0$ & $ H_1$ & $ H_2$ \\ \hline
$(1/2^{(8)})$\rule{0pt}{3ex} & $0$ & $\Z/2$ & $\Z^{90}$ \\
$(1/3^{(10})$ & $0$ & $\Z/2$ & $\Z^{650}$ \\
$(1/2^{(3)}, 1/6^{(6)})$\rule[-2ex]{0pt}{0pt} & $\Z^2$ & $\Z^{30}$
\end{tabular}
\captionof{table}{$\Delta_{0,w}$ may be disconnected or have torsion in $\pi_1$} \label{tab:pathological} 
\end{center}
It can be shown by hand that in the final example in table \ref{tab:pathological} above, $\Delta_{0,w}$ is homotopic to $\bigsqcup_{j=1}^3\term*{\bigvee_{i=1}^{10} S^1}$.

\section{Contractible subcomplexes and path spaces}
We begin our study of $\Delta_{0,w}$ by identifying the existence of a large contractible subcomplex $X_{0,w} \subset \Delta_{0,w}$ whose complement has a straightforward combinatorial description. Let $T \in \Delta_{0,w}$ with marking function \[ m_T : \set{1,\dots,n} \to V(T). \] Given $A \subset \set{1,\dots,n}$ we define the \textbf{weight} of $A$ to be
\[ w(A) = \sum_{i \in A} w_i. \]
Define the $A$-\textbf{marking locus} $U_A \subset \Delta_{0,w}$ to be the set of trees $T \in \Delta_{0,w}$ for which $m_T(A) = \set{v}$. That is, the set $U_A$ consists of trees for which all labels indexed by $A$ mark the same vertex.

We generalize the arguments of Robinson and Whitehouse \cite{robinson1996tree} to prove the following.
\begin{prop}\label{Contractible subcomplex} Let $\mathcal A \subset 2^{\set{1,\dots,n}}$ such that for all $A \in \mathcal A$, $w(A) > 1$ and $w(A^c) > 1$. Then $\bigcap_{A \in \mathcal A} U_A$ is either empty or contractible.
\end{prop}
\begin{proof}
Let $\mathcal{A}\subset 2^{\set{1,\dots,n}}$ be such a family, and assume the intersection $\bigcap_{A \in \mathcal{A}} U_A$ is nonempty. These conditions guarantee that the intersection $\bigcap_{A \in \mathcal{A}} U_A$ contains a facet $\Delta_k$ of maximal dimension. Let $T \in \text{int}(\Delta_k)$; then, $T \in U_B$ for some $B \in \mathcal{A}$. The tree $T$ has a vertex $v$ which is marked by $B$. Denote by $C\supset B$ the full set of markings supported at $v$. Contract $T$ to the tree in $U_B$ with two vertices -- one marked by $C$, the other by $C^c$. Denote this vertex of $\Delta_{0,{{w}}}$ by $v_C$. Recall that the closed star of a vertex $v$ is the union of all closed simplices sharing that vertex, and is denoted $\cl \St(\set{v})$. We claim that $\bigcap_{A \in \mathcal{A}} U_A$ is a conical subset of the closed star $\cl \St(\set{v_C})$.

To show this, we argue that $\bigcap_{A \in \mathcal{A}} U_A$ is a union of simplices in $\cl\St(\set{v_C})$ that meet at $v_C$. If $T \in \bigcap_{A \in \mathcal{A}} U_A$, it is also in the interior of some simplex; the boundaries of this simplex are determined by edge contraction, and thus also maintain $A$-marked vertices for all $A \in \mathcal{A}$. Consequently, $\bigcap_{A \in \mathcal{A}} U_A$ is a union of simplices. It remains to be shown that every facet in this intersection has $v_C$ as a vertex; this will imply that $\bigcap_{A \in \mathcal{A}} U_A$ is a union of simplices that meet at $v_C$. We will conclude that $\bigcap_{A \in \mathcal{A}} U_A$ is conical and hence contractible.

Consider a facet $\Delta_l \in \bigcap_{A \in \mathcal{A}} U_A$ such that $\Delta_l \ne \Delta_k$, and let $S \in \text{int}(\Delta_l)$. Let $u$ denote the vertex of $S$ marked by $D \supset B$. We can assume by maximality that $D$ marks a leaf vertex. Suppose for the sake of a contradiction that $D \ne C$. This implies that any $E \in \mathcal{A}$ such that $E \supset B$ must be contained in $C \cap D$. To see this, first consider $S$; the only vertex that can be marked by $E \supset B$ is $u$, and hence $E \subset D$. A similar argument on $T$ demonstrates that $E \subset C$, and thus $E \subset C \cap D$. Consequently, if $C \cap D$ is a proper subset of $D$ then we resolve the vertex $u$ supporting $C\cap D$ to two vertices $u_1$ and $u_2$, such that $C\cap D$ are the only markings on $u_1$. Note that by weight considerations, $u_1$ can be chosen to be a leaf vertex. This procedure furnishes a tree in $\bigcap_{A \in \mathcal{A}} U_A$  that contracts onto $S$, contradicting the fact that $S$ is in the interior of a facet in $\bigcap_{A \in \mathcal{A}} U_A$. We conclude that $D$ is contained in  $C$. The same argument shows that $C$ is contained in $D$, so the sets coincide. Contracting all of $S$ other than the edge adjacent to $u$, we see that $\Delta_l$ has $v_C$ as a vertex.  Thus if $\bigcap_{A \in \mathcal{A}} U_A$ is nonempty, it is contractible.
\end{proof}
Coupled with the following, we can construct a large contractible subcomplex of $\Delta_{0,w}$.

\begin{prop} \label{hatcher technicality}
Let $X = X_1 \cup \cdots \cup X_n$ be a union of $n \ge 1$ contractible CW-subcomplexes such that the intersection of any subfamily is contractible. Then $X$ is contractible
\end{prop}
\begin{proof}
	The case $n = 2$ is \cite[Exercise 0.23]{hatcher2002algebraic}, and the general statement follows by induction.
\end{proof}

Let $w = (w_1,\dots,w_n)$ be a weight vector with $w_1 + w_2 > 1$. Define the \textbf{heavy marking locus} to be the subset $X_{0,w} \subset \Delta_{0,w}$ consisting of trees $T$ having some vertex $v$ with $w(m_T^{-1}(v) \setminus \set{1,2}) > 1$.

\begin{prop}\label{Heavy marking locus}
If $w$ is as above and $\sum_{i=3}^{n}w_i>1$ then $X_{0,w}$ is contractible.
\end{prop}
\begin{proof}
Consider the family $\mathcal A \subset 2^{\set{3,\dots,n}}$ consisting of sets $A$ with $w(A) > 1$. Automatically, this ensures $w(A^c) \ge w(\set{1,2}) > 1$. Then $\bigcup_{A \in \mathcal A} U_A$ is the set of trees $T$ belonging to some $U_A$, i.e. trees $T$ for which there is $v \in V(T)$ with $m_T(A) = \set{v}$ for some $A \in \mathcal A$. But this is precisely the condition that there is $v$ for which $m_T^{-1}(v) \setminus \set{1,2}$ has weight $> 1$.

By Proposition \ref{Contractible subcomplex}, arbitrary intersections of $\set{U_A}_{A \in \mathcal A}$ are contractible. Applying Proposition \ref{hatcher technicality}, we see that $\bigcup_{A \in \mathcal A} U_A = X_{0,w}$ is contractible.
\end{proof}

The complement $\Delta_{0,w}\setminus X_{0,w}$ has a simple combinatorial description.

\begin{defn}
A \textbf{path space} is a space of tropical genus $0$ curves consisting only of paths, i.e. trees having exactly two leaves.
\end{defn}

\begin{lemma}\label{genus 0 post contraction}
Let $w = (1,1,w_3,\dots,w_{n})$. Then $\Delta_{0,w} \setminus X_{0,w}$ is a path space, whose elements consist of stable trees with one leaf marked by $(1)$, the other by $(2)$, and for which $w(m_T^{-1}(v) \setminus \set{1,2}) \le 1$ for all $v \in V(T)$.
\end{lemma}

\begin{proof}
If $T \in \Delta_{0,w}$ has any vertex $v$ marked by some $A = \set{i_1,\dots,i_k} \subset \set{3,\dots,n}$ with $w(A) > 1$ then $T \in U_{\set{i_1,\dots,i_k}} \subset X_{0,w}$. If $T$ has three leaves, then some leaf $v$ is marked by neither $(1)$ nor $(2)$. But $v$ must be stable, so there must be such an $A$ marking $v$. We conclude that $T$ has exactly $2$ leaves.
\end{proof}

\begin{rmk}
If $\Delta_{0,w}$ is connected and $X_{0,w}$ is nonempty then $\Delta_{0,w}$ has the homotopy type of the one-point compactification of $\Delta_{0,w} \setminus X_{0,w}$. This follows from the universal property of topological quotients since $\Delta_{0,w}$ is compact and Hausdorff and $X_{0,w}$ is a contractible subcomplex.
\end{rmk}

\section{The topology of $\Delta_{0,w}$}\label{sec:structural}
We reduce Theorem \ref{A} to a statement about the topological combinatorics of a particular subspace arrangement. This arrangement's intersection with a sphere at the origin will be shown to in Section \ref{Homotopy Type of the subspace arrangement} to have the homotopy type of a wedge of spheres, from which the theorem is deduced.

\begin{figure}
\begin{tikzpicture}[line cap = round, line join = round]
\coordinate (lefty) at (-0.5,0);
\coordinate (righty) at (0.5,0);
\coordinate (rvert) at (1.5,0);
\coordinate (lvert) at (-1.5,0);
\coordinate (rleaf) at (2.4659258263, 0.2588190451);
\coordinate (lleaf) at (-2.4659258263, 0.2588190451);
\node (l)  at (-2.9659258263, 1.1248444489) {$\color{Firebrick4}(1)$};
\node (x1) at (-3.5659258263, 0.2588190451) {$(i_1)$};
\node (x2) at (-2.9659258263, -0.6072063587) {$(i_2)$};
\node (x3) at (-1.6305261922, -0.9914448614) {$(i_3)$};
\node (x4) at (1.6305261922, -0.9914448614) {$(i_{n-4})$};
\node (x5) at (2.9659258263, -0.6072063587) {$(i_{n-3})$};
\node (x6) at (3.5659258263, 0.2588190451) {$(i_{n-2})$};
\node (r)  at (2.9659258263, 1.1248444489) {$\color{DodgerBlue4}(2)$};

\draw[darkgray,dotted] (lefty) -- (righty);
\draw (lefty) -- (lvert) -- (lleaf)
      (righty) -- (rvert) -- (rleaf);

\draw[gray,dashed] (lleaf) -- (l)
                   (lleaf) -- (x1)
                   (lleaf) -- (x2)
                   (lvert) -- (x3)
                   (rvert) -- (x4)
                   (rleaf) -- (x5)
                   (rleaf) -- (x6)
                   (rleaf) -- (r);

\end{tikzpicture}
\caption{A pathlike tree}
\label{fig:pathlike}
\end{figure}
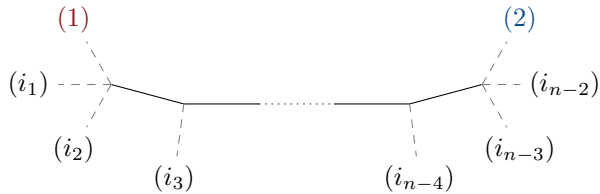

	\begin{numthm}{A}
	Let $w = (1,1,w_3,\dots,w_{n})$. Then $\Delta_{0,w}$ has the homotopy type
	\[ \Delta_{0,w} \simeq S^{n-4} \vee \bigvee_{i=1}^r S^{d_i} \]
	for some list of integers $0 \le d_i \le n-4$.
	\end{numthm}

We assume for the rest of this section that $w=(1,1,w_3,\dots, w_n)$ and exclude the following edge cases. Either $\sum_{i=3}^{n} w_i \le 1$, in which case $X_{0,w}$ is empty and $\Delta_{0,w} \cong S^{n-4}$, or else $X_{0,w}$ is nonempty. If $w$ is such that $\sum_{i=3}^{n} w_i > 1$ but for every $A_t = \set{3,\dots,n} \setminus \set{t}$ we have $\sum_{i \in A_t} w_i \le 1$, then $X_{0,w}$ is an isolated point and $\Delta_{0,w} \cong S^{n-4} \vee S^0$. This occurs, for example, in the case $w = (1, 1, \frac1m^{(m+1)})$ -- see Section \ref{sec:examples}.  For all possible values of $w$ with $X_{0,w}$ nonempty, either $X_{0,w}$ is an isolated point or $\Delta_{0,w}$ is connected.

Disregarding these cases, we are in the situation of Lemma \ref{genus 0 post contraction}. We reserve the names $(1)$ and $(2)$ for the two distinguished heavy markings and use a nonstandard indexing for the remaining entries of $w$.

\begin{prop}\label{prop:global path space}
We define\edit{P6Q3: clarified our use of the term ``global path space" as a name for the set $S$.} the \textbf{global path space} by
\[
	S = \set{(x_3, \dots, x_n) \in [0,1]^{n-2} \mid x_i = 0,\, x_j = 1 \text{ for some $i,j$}} \subset \R^{n-2}.
\]
There is an injection $p : S \to \Delta_{0,w}$ whose image contains $\Delta_{0,w} \setminus X_{0,w}$.
\end{prop}
\begin{proof}
The map $p$ sends a tuple $(x_3,\dots,x_n)$ to a pathlike marked tree with $(1)$ on one leaf, $(2)$ on the other, and mark $(i)$ at distance $x_i$ from $(1)$. Such trees are automatically stable: each internal vertex has edge degree 2 and supports a positive weight; each leaf has edge degree 1 and supports a heavy weight, as well as some positive weight. The map $p$ is an injection since it has a well-defined inverse. Every tree in $\Delta_{0,w} \setminus X_{0,w}$ is pathlike with marks $(1)$ and $(2)$ at the leaves, and thus in the image of $p$. \end{proof}

Note that $p$ is an injection between compact, Hausdorff spaces, and thus is a homeomorphism to its image. Thus $S \cong p(S)$, and so $S \setminus p^{-1}(X_{0,w}) \cong p(S) \setminus X_{0,w} = \Delta_{0,w} \setminus X_{0,w}$. Thus $\Delta_{0,w} / X_{0,w} \cong S / p^{-1}(X_{0,w})$. The set $p^{-1}(X_{0,w})$ is characterized as follows. Let $\mathcal A = \set{A \in 2^{\set{3, \dots, n}} \mid \sum_{i \in A} w_i > 1}$. This yields a subspace arrangement
\[
	D = \bigcup_{A \in \mathcal A} \bigcap_{i,j \in A} \set{x \in \R^{n-2} \mid x_i = x_j} \subset \R^{n-2}.
\]
See Section \ref{Homotopy Type of the subspace arrangement} for a discussion of these induced subspace arrangements. For each $x \in S$, $x_i = x_j$ if and only if the marks $(i)$ and $(j)$ coincide at some vertex of $p(x)$. Thus $p^{-1}(X_{0,w})$ is precisely $S \cap D$.

\begin{prop}\label{wedge of lemmas}
There is a homeomorphism $\Susp (S \cap D) \cong \partial[0,1]^{n-2} \cap D$.
\end{prop}

\begin{proof}
Let $\Susp(S \cap D)$ be the suspension $((S \cap D) \times [-1,1])/\sim$. There is a continuous $\psi : (S \cap D) \times [-1,1] \to \partial[0,1]^{n-2}$ that is constant on $(S \cap D) \times \set{-1}$ and $(S \cap D) \times \set{1}$. Explicitly this map is defined in terms of the convex combinations
\[
\psi(x,t) = \begin{cases} \, \hfill \, t \cdot \mathbf 1 + (1-t) \cdot x  \, \hfill \, & \text{if } t \ge 0 \\ \, \hfill \, t \cdot \mathbf 0 + (1+t) \cdot x  \, \hfill \, & \text{if } t < 0,  \end{cases}
\]
where $\mathbf 0 = (0,\dots,0)$ and $\mathbf 1 = (1,\dots,1)$. If $x \in S \cap D$ then $x \in H \subset D$ for some linear subspace $H$. Since $H \supset \set{\mathbf 0, \mathbf 1}$ and $H$ is convex, we are ensured $\psi(x,t) \in H$ for all $t$. Similarly, since the faces of $\partial[0,1]^{n-2}$ are convex and contain $\set{\mathbf 0, \mathbf 1}$, we are ensured $\psi(x,t) \in \partial[0,1]^{n-2}$ for all $t$. Finally, we show $\psi$ is injective restricted to $(S \cap D) \times (-1,1)$. Suppose $\psi(x,t) = \psi(z,s)$. Then $t$ and $s$ have the same sign, or are both zero, since $\psi((S \cap D) \times (0,1))$ and $\psi((S \cap D) \times (-1,0))$ are disjoint. By symmetry we can assume $t,s \le 0$. Then $(1 + t) \cdot x = (1 + s) \cdot z$, and so
\[ x = \frac{1 + s}{1 + t} \cdot z. \]
Since $x$ has some coordinate $x_i = 1$ we cannot have $(1 + t)/(1 + s) > 1$ or else $z \notin [0,1]^{n-2}$. Symmetrically, $z$ has a coordinate $z_j = 1$ and so we cannot have $(1 + s)/(1 + t) > 1$ or else $x \notin [0,1]^{n-2}$. We conclude $s = t$, in which case $x = z$. Thus $\psi$ descends to an embedding $\overline \psi : \Susp(S \cap D) \hookrightarrow \partial[0,1]^{n-2} \cap D$.

Now we show $\overline \psi$ is surjective. Let $y \in \partial[0,1]^{n-2} \cap D$. We may assume $y \notin \set{\mathbf 0, \mathbf 1}$ since both $\mathbf 0$ and $\mathbf 1$ are easily seen to be in the image of $\overline \psi$. Similarly we may assume $y \notin S \cap D$. Since $y \in \partial[0,1]^{n-2}$, we have $\max \set{y_i} = 1$ or $\min \set{y_i} = 0$, but not both.
\begin{enumerate}
\item If $\max \set{y_i} = 1$ and $\min \set{y_i} < 1$, then $x = (1 - \min\set{y_i})^{-1}\cdot (y - \min\set{y_i} \cdot \mathbf 1)$ has $\min\set{x_i} = 0$ and $\max \set{x_i} = 1$, so $x \in S \cap D$. And $\psi(x, \min \set{y_i}) = y$.
\item Similarly, if $\max \set{y_i} > 0$ and $\min \set{y_i} = 0$, then $x = (\max\set{y_i})^{-1} \cdot y$ satisfies $\min\set{x_i} = 0$ and $\max \set{x_i} = 1$, so $x \in S \cap D$. And $\psi(x, \max\set{y_i}-1) = y$.
\end{enumerate}
This shows that $\overline \psi$ surjects on $\partial[0,1]^{n-2} \cap D$, so we conclude $\Susp(S \cap D) \cong \partial[0,1]^{n-2} \cap D$.
\end{proof}

As a result of the proof of Proposition \ref{wedge of lemmas}, we have the following identification.

\begin{lemma}
The universal path space $S \cong S^{n-4}$.
\end{lemma}
\begin{proof}
We can recognize $S$ as a retract of $\partial[0,1]^{n-2} \setminus \set{\mathbf 0, \mathbf 1}$, and this retraction map is precisely the inverse of $\psi$ constructed in the proof of Proposition \ref{wedge of lemmas}.
\end{proof}

We now invoke the following standard fact. \edit{Changed X and Y to be CW complexes, not just topological spaces.}
\begin{lemma}
Let $X$ be a CW complex and $Y \subset X$ a subcomplex whose inclusion $\iota : Y \hookrightarrow X$ is nullhomotopic. Then $X/Y \simeq X \vee \Susp Y$.
\end{lemma}

\begin{cor}\label{cor: reduction-to-subspaces}
Let $w = (1,1,w_3,\ldots, w_n)$. Then
\[
\Delta_{0,w}\simeq S^{n-4}\vee \left( \partial [0,1]^{n-2}\cap D\right).
\]
\end{cor}
\begin{proof} The inclusion $S \cap D \hookrightarrow S$ is nullhomotopic because each linear subspace $H \subset D$ has positive codimension and $S \cong S^{n-4}$ is $(n-5)$-connected. Thus $S / (S \cap D) \simeq S \vee \Susp(S \cap D) \cong S \vee (\partial[0,1]^{n-2} \cap D)$. \end{proof}

\begin{figure}
\begin{tikzpicture}[line cap = round, line join = round]

\draw[style = dotted, opacity = 0.3] (0,0) -- (0,2)
                      (0,0) -- (2,0)
                      (0,0) -- (1,0.5);

\draw[Firebrick4] (1.9,0.5) -- (1,0.5) -- (1,1.9)
      (1,2.1) -- (1,2.5) -- (0,2)
      (3,0.5) -- (2.1,0.5);

\fill[DodgerBlue1, opacity= 0.2] (0,0) -- (2,0) -- (3,2.5) -- (1,2.5);
\draw[DodgerBlue4, dashed] (0,0) -- (2,0) -- (3,2.5) -- (1,2.5) -- (0,0);

\draw[style=dotted,opacity= 0.3]   (3,2.5) -- (3,0.5)
                      (3,2.5) -- (1,2.5)
                      (3,2.5) -- (2,2);

\draw[Firebrick4] (0,2) -- (2,2) -- (2,0) -- (3,0.5);

\draw [fill = white] (2,0) ellipse (0.05 and 0.05)
              (1,2.5) ellipse (0.05 and 0.05);

\node (O) at (-0.2,0) {$\mathbf 0$};
\node (W) at (3.2,2.5) {$\mathbf 1$};
\node (S) at (0.5, 2.5) {$\color{Firebrick4}{S}$};
\node (D) at (2.75, 1.25) {$\color{DodgerBlue4}{D}$};
\end{tikzpicture}

\caption{The universal path space $S$ and subspace arrangement $D$ when $w = (1,1,\frac23,\frac12,\frac13)$. In this case $S \cong S^1$, $D$ is a single plane $\set{x_3 = x_4} \subset \R^3$, and $S \cap D$ consists of two points.}
\label{fig:squiggle}
\end{figure}
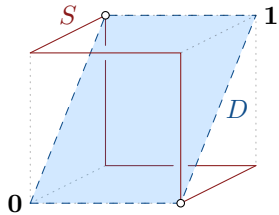

Figure \ref{fig:squiggle} gives a schematic for the equivalence of $\Susp(S \cap D)$ and $\partial[0,1]^{n-2} \cap D$.

\begin{rmk}\label{rmk: round sphere} The space $\partial[0,1]^{n-2} \cap D$ is homeomorphic to $S^{n-3} \cap D$, where $S^{n-3}$ is the standard \edit{P7Q4: removed the word ``round".}sphere in $\R^{n-2}$.
\end{rmk}
\begin{proof} Each linear subspace $H \subset D$ contains the line $\set{x_3 = \cdots = x_n} = \vspan \set{(1,\dots,1)}$, and so $D$ is invariant under translation by $(1,\dots,1)$. Thus $\partial[0,1]^{n-2} \cap D \cong \partial[-1/2,1/2]^{n-2} \cap D$. The homeomorphism  \[ \partial[-1/2,1/2]^{n-2} \to S^{n-3} \quad \text{sending} \quad x \mapsto x/\norm x \] preserves linear subspaces, and so its restriction to $\partial[-1/2,1/2]^{n-2} \cap D \to S^{n-3} \cap D$ is also a homeomorphism.
\end{proof}
In Section \ref{Homotopy Type of the subspace arrangement} we finish the proof of Theorem~\ref{A} by studying the topology of $S^{n-3} \cap D$.

\section{Homotopy type of the subspace arrangement}\label{Homotopy Type of the subspace arrangement}

In this section, we complete the proof of Theorem~\ref{A}. By Corollary~\ref{cor: reduction-to-subspaces} and Remark~\ref{rmk: round sphere}, it suffices to determine the homotopy type of the $S^{n-3} \cap D$. We work with the concept of shellability of nonpure simplicial complexes studied by Bj\"orner and Wachs, and refer the reader to~\cite[Section 2]{bjornerwachsnonpureshellability} for further details on the subject.

\begin{defn}
An ordering on the maximal facets $C_i$ of a simplicial complex is a \textbf{shelling} of the complex if for every $k>1$ and $a<k$, there is some $b<k$ and $x \in C_k$ such that $C_a \cap C_k \subset C_b \cap C_k =C_k \setminus \set x$.
\end{defn}
Recall that if a nonpure simplicial complex is shellable, it has the homotopy type of a wedge of spheres of possibly varying dimensions.

\begin{thm}\label{subspace spheres}
Let $w = (1,1,w_3,\dots,w_n) \in (0,1]^n$ be given. Let $D$ be the subspace arrangement in $\R^{n-2}$ induced by heavy subsets $A \subset \set{3,\dots,n}$, as described in Section \ref{sec:structural}. Then the \textbf{singularity link} $L(D) = S^{n-3} \cap D$ has the homotopy type of a wedge of spheres.
\end{thm}

Subspace arrangements whose constituent subspaces take the form $\set{x_{i_1} = \cdots = x_{i_k}}$ are known as \textbf{diagonal arrangements}. There is a natural bijection between diagonal arrangements in $\R^n$ and simplicial complexes of dimension at most $n-3$ on the vertex set $\set{1,\dots,n}$. Given a simplicial complex $\Delta$, the correspondence associates to each face $F \in \Delta$ the subspace $H_F = \set{x \in \R^n \mid x_{i_1} = \cdots = x_{i_k}}$ where $\set{i_1,\dots,i_k} = \set{1,\dots,n} \setminus F$. Then $\Subs_\Delta = \bigcup_{F \in \Delta} H_F$ is the corresponding diagonal arrangement.

\begin{thm}[{\cite[Theorem~1.1 and Corollary~5.2]{kimshellablelattice}}]
If $\Delta$ is shellable, then the singularity link $L(\Subs_\Delta)$ has the homotopy type of a wedge of spheres.
\end{thm}

\begin{defn}
Let $u = (u_1,\dots,u_r) \in (0,1]^r$, and let $\mathcal A = \set{A \in 2^{\set{1,\dots,r}} \mid \sum_{i \in A} u_i > 1}$. Then we define the \textbf{$u$-induced subspace arrangement}
\[ \Subs_u = \bigcup_{A \in \mathcal A} \bigcap_{i,j \in A} \set{x_i = x_j} \subset \R^r, \]
and the \textbf{$u$-induced simplicial complex} $\Delta_u = \set{A^c \mid A \in \mathcal A}$.
\end{defn}
Our definition is such that $\Delta_u$ and $\Subs_u$ are in correspondence via the bijection in \cite{kimshellablelattice}. The arrangement $D$ in Section \ref{sec:structural} is the $(w_3,\dots,w_n)$-induced subspace arrangement in $\R^{n-2}$. It will suffice to show $\Delta_u$ is shellable to prove Theorem \ref{A}.

Note that $\dim \Delta_u \le r-3$. Since each $u_i \le 1$, $\abs F \ge 2$ for all $F \in \Delta_u$ and hence $\abs{F^c} \le r-2$. Thus each facet of $\Delta_u$ has no more than $r-2$ vertices and dimension no greater than $r-3$.

\begin{prop}\label{shellability of the associated complex}
Let $u = (u_1,\dots,u_r) \in (0,1]^r$. Then $\Delta_u$ is a shellable simplicial complex.
\end{prop}

\edit{P8Q4: defined the partial order $\prec$ and clarified this discussion.}
\begin{proof}
Without loss of generality, we assume $u_1 \le u_2 \le \cdots \le u_r$, so that $i \le j$ if and only if $u_i \le u_j$. A facet $C \in \Delta_u$ is determined uniquely by the tuple $(i_1,\dots,i_k) \in \set{1,\dots,r}^k$ of its vertices with $i_1 < \cdots < i_k$. Let ``$\prec$'' be the lexicographic order on these tuples. Even though these simplices are of varying dimensions, there is no ambiguity in the order as the intersections are necessarily proper subsets of the intersecting simplices. We label the facets of $\Delta_u$ by $\set{C_a : a \in \N}$ so that $C_a \prec C_b \iff a < b$.

Consider the facet $C_k$ for $k>1$ and let $a<k$. As the ordering is lexicographic, there is some critical coordinate at which the entry in the tuple corresponding to $C_k$ is bigger than that of $C_a$. In other words, we have
\[
C_a=(A,x_c,C) \prec (A,y_c,B) = C_k,
\]
where $A$ is possibly empty and $x_c<y_c$. The intersection $C_a \cap C_k$ is some face contained in $C_k$, possibly of codimension $>1$. As $C_a$ is not contained in $C_k$, there is some $z \in C_k$ such that $z \notin C_a$. Denote by $z_c$ the minimal such $z$ (with respect to the order on integers), and note that it may be $y_c$. By construction, $x_c < z_c$ and $x_c \notin C_k$. Consider the simplex of $\Delta_u$ corresponding to the tuple by $(C_k \setminus \set{z_c}) \cup \set{x_c}$, reordered. This is a simplex in $\Delta_u$, as
\[
\sum \set{ u_i \mid i \notin (C_k \setminus \set{z_c}) \cup \set{x_c}} \ge \sum \set{u_i \mid i \notin C_k} > 1
\]
as $x_c \le z_c$. Let $C_b$ denote the minimal facet with respect to the lexicographic ordering on facets induced by $\prec$ containing $(C_k \setminus \set{z_c}) \cup \set{x_c}$. Then $C_b \cap C_k=C_k \setminus \set{z_c}$, as $((C_k \setminus \set{z_c}) \cup \set{x_c}) \cap C_k = C_k \setminus \set{z_c}$ and $C_b$ cannot contain $z_c$ as it would then contain $C_k$ as a proper subset, a contradiction. Furthermore, $C_a \cap C_k \subset C_b \cap C_k$ since $C_b \cap C_k \supset C_k \setminus \set{z_c} \supset C_a \cap C_k$.

Finally, we show $b<k$. First, if $C_b=(C_k \setminus \set{z_c})\cup \set{x_c}$ then $b<k$ trivially, as the tuples corresponding to $C_b$ and $C_k$ first differ at the entry at which $C_b$ has an $x_c$ and $C_k$ has a $y_c$. Otherwise, $C_b$ has as a vertex some index $v$ not in $((C_k \setminus \set{z_c})\cup \set{x_c})$. If $x_c \le v$, then $C_b$ and $C_k$ first differ at the entry at which $C_b$ has an $x_c$ and $C_k$ has a $y_c$ as before and $b<k$. Otherwise, $C_b$ and $C_k$ first differ at the entry at which $C_b$ has a $v$ and $C_k$ has the minimal index in $C_b$ greater than $v$, and so $b<k$. Regardless, $b<k$ and $C_a \cap C_k \subset C_b \cap C_k =C_k \setminus \set{z_c}$, and the lexicographic ordering on facets is a shelling of $\Delta_u$.
\end{proof}

We now prove Theorem \ref{subspace spheres}.

\begin{proof}[Proof of Theorem \ref{subspace spheres}]
By Proposition \ref{shellability of the associated complex}, the simplicial complex associated to $D = \Subs_{(w_3,\dots,w_n)}$ is shellable. So the singularity link $L(D) = S^{n-3} \cap D$ is homotopic to a wedge of spheres by~\cite[Corollary~5.2]{kimshellablelattice}.
\end{proof}
By Remark \ref{rmk: round sphere}, $L(D) \cong \partial[0,1]^{n-2} \cap D$, so this finishes the proof of Theorem \ref{A}.

\section{Homology of $\Delta_{0,w}$: Formulas}\label{sec: heavy-light}

When the weight vector has the form $(1,1,w_3,\dots,w_n)$, Theorem~\ref{A} reduces the determination of the homotopy type of $\Delta_{0,w}$ to a calculation of Betti numbers. A result of Goresky and MacPherson \cite[Part III, Chapter 1]{goresky1988stratified} computes the homology of the link $L(D)$ for a subspace arrangement $D \subset \R^n$ in terms of its intersection lattice $\mathcal L$. If $d : \mathcal L \rightarrow \Z_{\ge 0}$ is the function giving the dimension of each subspace $p \in \mathcal L$ and $\Delta(\mathcal L_{< p})$ is the order complex of the restriction of the lattice to elements less than $p$, one obtains the following formula:
\[
\widetilde{H}_k(L(D)) \cong \bigoplus_{p \in \mathcal L}\widetilde{H}_{k-d(p)}(\Delta(\mathcal L_{<p})).
\]

\begin{thm}
Let $w=(1,1,w_3,\dots,w_n)$ and $D \subset \R^{n-2}$ be as in the proof of Theorem \ref{A}. The reduced homology of $\Delta_{0,w}$ can be computed as follows:
\[
\widetilde{H}_k(\Delta_{0,w}; \Z) \cong \begin{cases}
~0 & \text{ if } k > n-4; \\
~\bigoplus_{p \in \mathcal L}\widetilde{H}_{k-d(p)}(\Delta(\mathcal L_{<p})) \oplus \Z & \text{ if } k=n-4; \\
~\bigoplus_{p \in \mathcal L}\widetilde{H}_{k-d(p)}(\Delta(\mathcal L_{<p})) & \text{ otherwise.}
\end{cases}
\]
Furthermore, these homology groups are free abelian.
\end{thm}
\begin{proof}
The simplicial complex $\Delta_{0,w}$ is $(n-4)$-dimensional, so homology vanishes in dimensions $k > n-4$. As established in Corollary \ref{cor: reduction-to-subspaces} and Remark \ref{rmk: round sphere}, there is a subspace arrangement $D \subset \R^{n-2}$ such that
\[
\Delta_{0,w} \simeq S^{n-4} \vee L(D).
\]
This, together with Goresky and MacPherson's formula, yields our homology computation. Since $L(D)$ has the homotopy type of a wedge of spheres, these homology groups are necessarily free abelian. \end{proof}

\begin{rmk}
If $w = (1^{(k)},w_{k+1},\dots,w_n)$, then $(k-2)!$ divides every reduced homology Betti number of $\Delta_{0,w}$. Indeed, there is an $S_{k-2}$ action on the space $\R^{n-2} \setminus D$ induced by permuting marks $\set{3,\dots,k}$. This action is free on $\R^{n-2} \setminus D$ and one can observe that there is a conical fundamental domain $C$ for this action. Standard arguments imply the claimed divisibility.
\end{rmk}

We now specialize to the weight data $w = (1,1,\epsilon^{(k)})$ for $\epsilon = 1/\ell$. The $(w_3,\dots,w_n)$-induced subspace arrangement consists of subspaces of dimension $\ell + 1$. Bj\"{o}rner and Welker \cite{bjorner1995homology} derive formulas for the homology of the intersection lattices of such $(\ell+1)$-equal arrangements. These formulas are quite complex, so we refer the reader to the original source.
The particular homology calculation can be deduced from Alexander duality and the formula in~\cite[Theorem 5.2]{bjorner1995homology}. 

\begin{thm}\label{thm:gaps}
Let $w = (1,1,\epsilon^{(k)})$ with $\epsilon = 1/\ell$. Then, $\widetilde{H}_d \ne 0$ if and only if $d=n - 4 - t ( \ell - 1)$ for $0 \le t \le \floor*{(n-2)/(l+1)}$.
\end{thm}

Theorem \ref{thm:gaps} accounts for patterns in the Betti tables recorded in Section \ref{sec:examples}.

\subsection{Heavy/light weight data} Specializing further and decreasing the parameter $\epsilon$, we obtain the heavy/light moduli space which features prominently in~\cite{cavalieri2016moduli}. It can be deduced from the results in loc.~cit.~that the spaces $\Delta_{0,w}$ are homotopic to wedge sums of top dimensional spheres, although the fact is not explicitly recorded there. We give an independent proof of this fact below, and obtain a closed formula for the number of spheres.

\begin{thm}\label{heavy light}
Let $w = (1^{(m)},\epsilon^{(k)})$ for $0<\epsilon \ll 1$. Then $\Delta_{0,w}$ has the homotopy type of a wedge of $(m-2)!(m-1)^k$ spheres of dimension $(m+k-4)$.
\end{thm}

\begin{proof}
The proof is by induction on $k$. Our base case is Vogtmann's result that $\Delta_{0,(1^{(m)})} \simeq \bigvee_{(m-2)!} S^{m-4}$ in \cite{vogtmannouterspace}.
In addition, we take as separate base cases $\Delta_{0,(1^{(3)}, \epsilon)} \simeq \bigvee_{2} S^0$ and $\Delta_{0,(1^{(2)}, \epsilon^{(2)})} \simeq S^0$.

Let $w = (1^{(m)}, \epsilon^{(k)})$ and $w' = (1^{(m)}, \epsilon^{(k+1)})$. We will build $\Delta_{0,w'}/X_{0,w'}$ by attaching an $\epsilon$-mark to $\Delta_{0,w}/X_{0,w}$. Define
\[V = [1,m] / \set{1,\dots,m} \cong \bigvee_{i=1}^{m-1}S^1,\] and let $a$ denote the basepoint. We will build a map $\varphi : S_w / B_w \times V \to S_{w'}/B_{w'}$. Recall that we use the nonstandard indexing $(x_3,\dots,x_{m+k})$ for points in $S_w$. We define $\varphi$ by specifying $y = \varphi(x,p)$ piecewise.
\begin{enumerate}
\item[\bf 1.] If $x \in B_w$, let $y \in B_{w'}$ for all $p \in V$.
\item[\bf 2.] Similarly, if $p = a$, let $y \in B_{w'}$ for all $x \in S_w / B_w$.
\item[\bf 3.] If $x \in S_w \setminus B_w$ and $p \neq a$, then the coordinates $\set{x_3,\dots,x_m}$ are distinct, and so there is an index set $\set{i_1,\dots,i_{m-2}}$ for which $x_{i_1} < x_{i_2} < \cdots < x_{i_{m-2}}$. Furthermore, $p \in [1,m] \setminus \set{1,\dots,m}$, so $p \in (j,j+1)$ for some $j \in \set{1,\dots,m-1}$.
\begin{enumerate}
\item If $j \notin \set{1,m-1}$, let $y_i = x_i$ for $3 \le i \le m + k$ and $y_{m+k+1} = (1-(p-j))x_{i_j} + (p-j)x_{i_{j+1}}$.
\item If $j = 1$, let $t = x_{i_1} + 1 -(p-j)^{-1}$.
\begin{enumerate}
\item If $t \ge 0$, let $y_i = x_i$ for $3 \le i \le m + k$ and $y_{m+k+1} = t$.
\item If $t < 0$, let $y_i = 1 - (1-x_i)/(1-t)$ for $3 \le i \le m + k$ and $y_{m+k+1} = 0$.
\end{enumerate}
\item If $j = m-1$, let $t = x_{i_{m-2}} - 1 + (1-(p-j))^{-1}$.
\begin{enumerate}
\item If $t \le 1$, let $y_i = x_i$ for $3 \le i \le m + k$ and $y_{m+k+1} = t$.
\item If $t > 1$, let $y_i = x_i/t$ for $3 \le i \le m + k$ and $y_{m+k+1} = 1$.
\end{enumerate}
\end{enumerate}
\end{enumerate}
\edit{P10Q4: added explanation of continuity of $\varphi$.}
To see this definition makes $\varphi$ continuous, note that for any $j \in \set{1,\dots,m-1}$, our map $\varphi$ is continuous restricted to the open set
\[ (S_w\setminus B_w) \times (j,j+1) \subset (S_w \setminus B_w) \times (V \setminus \set a), \] 
and that for distinct $j$ these open sets are disjoint. These open sets comprise the complement of $(S_w / B_w \times \set a) \cup (\set{B_w} \times V)$ and are mapped under $\varphi$ into $S_{w'} \setminus B_{w'}$. Furthermore, $\varphi$ is constant restricted to $(S_w/B_w \times \set{a}) \cup (\set{B_w} \times V)$, and so descends to a map $\overline \varphi : (S_w / B_w) \wedge V \to S_{w'} / B_{w'}$. 

This $\overline \varphi$ is surjective because every tree in $\Delta_{0,w'} \setminus X_{0,w'}$ may be obtained either by attaching an $\epsilon$-mark to the interior of a pathlike tree in $\Delta_{0,w} \setminus X_{0,w}$, or by placing an $\epsilon$-mark at a leaf and redistributing the remaining marks. Furthermore $\overline \varphi$ is injective outside of $\varphi^{-1}(B_{w'})$ because removing mark $(m+k+1)$ from a tree in $\Delta_{0,w'} \setminus X_{0,w'}$ and redistributing other marks if necessary is a well-defined inverse map to $\Delta_{0,w} \setminus X_{0,w}$. Thus $\overline \varphi$ is a homeomorphism. We conclude that \[ \Delta_{0,w}/X_{0,w} \wedge \bigvee_{i=1}^{m-1} S^1 \cong \Delta_{0,w'}/X_{0,w'}. \]

By hypothesis $\Delta_{0,w}/X_{0,w} \simeq \bigvee_{N(m,k)}S^{m+k-4}$. Thus $\Delta_{0,w'}/X_{0,w'} \simeq \bigvee_{(m-1) \cdot N(m,k)}S^{m+k-3}$. It follows that $N(m,k+1) = (m-1)\cdot N(m,k)$. The result follows by induction.
\end{proof}

The idea of the proof is illustrated in Figure \ref{fig:smash}. Beginning with $S^1 \vee S^1$, we form the product with the interval, identify three levels at which to pinch, and contract the vertical segment. This gives us a wedge of $2 \cdot (3-1) = 4$ spheres of dimension $1+1=2$.
\begin{figure}
\begin{tikzpicture}[line cap = round, thin, scale=0.6]
\draw (0.5,4) + (0:1 and 0.25) arc (0:180:1 and 0.25)
      (2.5,4) + (0:1 and 0.25) arc (0:180:1 and 0.25)
      (0.5,4) + (0:1 and 0.25) arc (0:-180:1 and 0.25)
      (2.5,4) + (0:1 and 0.25) arc (0:-180:1 and 0.25);

\draw[style = dashed] (0.5,0) + (0:1 and 0.25) arc (0:180:1 and 0.25)
                      (2.5,0) + (0:1 and 0.25) arc (0:180:1 and 0.25);

\draw (0.5,0) + (0:1 and 0.25) arc (0:-180:1 and 0.25)
      (2.5,0) + (0:1 and 0.25) arc (0:-180:1 and 0.25);

\draw (-0.5,0) -- (-0.5,4)
      (1.5,0) -- (1.5,4)
      (3.5,0) -- (3.5,4);

\draw[style = dashed] (0.5,2) + (0:1 and 0.25) arc (0:180:1 and 0.25)
                      (2.5,2) + (0:1 and 0.25) arc (0:180:1 and 0.25);

\draw (0.5,2) + (0:1 and 0.25) arc (0:-180:1 and 0.25)
      (2.5,2) + (0:1 and 0.25) arc (0:-180:1 and 0.25);

\fill[lightgray,opacity = 0.3] (0.5,0) ellipse (1 and 0.25)
                               (2.5,0) ellipse (1 and 0.25)
                               (0.5,2) ellipse (1 and 0.25)
                               (2.5,2) ellipse (1 and 0.25)
                               (0.5,4) ellipse (1 and 0.25)
                               (2.5,4) ellipse (1 and 0.25);

\draw[->] (4.125,2) -- (5.125,2);

\fill[lightgray,opacity=0.3] (6.5,0) ellipse (0.5 and 0.125)
                             (7.5,0) ellipse (0.5 and 0.125)
                             (6.5,2) ellipse (0.5 and 0.125)
                             (7.5,2) ellipse (0.5 and 0.125)
                             (6.5,4) ellipse (0.5 and 0.125)
                             (7.5,4) ellipse (0.5 and 0.125);

\draw[style = dashed] (6.5,0) + (0:0.5 and 0.125) arc (0:180:0.5 and 0.125)
                      (7.5,0) + (0:0.5 and 0.125) arc (0:180:0.5 and 0.125)
                      (6.5,2) + (0:0.5 and 0.125) arc (0:180:0.5 and 0.125)
                      (7.5,2) + (0:0.5 and 0.125) arc (0:180:0.5 and 0.125);

\draw (6.5,4) + (0:0.5 and 0.125) arc (0:180:0.5 and 0.125)
      (7.5,4) + (0:0.5 and 0.125) arc (0:180:0.5 and 0.125);

\draw (6.5,0) + (0:0.5 and 0.125) arc (0:-180:0.5 and 0.125)
      (7.5,0) + (0:0.5 and 0.125) arc (0:-180:0.5 and 0.125)
      (6.5,2) + (0:0.5 and 0.125) arc (0:-180:0.5 and 0.125)
      (7.5,2) + (0:0.5 and 0.125) arc (0:-180:0.5 and 0.125)
      (6.5,4) + (0:0.5 and 0.125) arc (0:-180:0.5 and 0.125)
      (7.5,4) + (0:0.5 and 0.125) arc (0:-180:0.5 and 0.125);

\draw[style = dashed] (6.25,1) + (0:0.75 and 0.1875) arc (0:180:0.75 and 0.1875)
                      (7.75,1) + (0:0.75 and 0.1875) arc (0:180:0.75 and 0.1875)
                      (6.25,3) + (0:0.75 and 0.1875) arc (0:180:0.75 and 0.1875)
                      (7.75,3) + (0:0.75 and 0.1875) arc (0:180:0.75 and 0.1875);

\draw (6.25,1) + (0:0.75 and 0.1875) arc (0:-180:0.75 and 0.1875)
      (7.75,1) + (0:0.75 and 0.1875) arc (0:-180:0.75 and 0.1875)
      (6.25,3) + (0:0.75 and 0.1875) arc (0:-180:0.75 and 0.1875)
      (7.75,3) + (0:0.75 and 0.1875) arc (0:-180:0.75 and 0.1875);

\draw (6.75,1) + (126:1.25 and 1.28) arc (126:234:1.25 and 1.28)
      (6.75,3) + (126:1.25 and 1.28) arc (126:234:1.25 and 1.28)
      (7.25,1) + (-54:1.25 and 1.28) arc (-54:54:1.25 and 1.28)
      (7.25,3) + (-54:1.25 and 1.28) arc (-54:54:1.25 and 1.28);

\draw (7,0) -- (7,4);

\draw[->] (8.75,2) -- (9.75,2);

\draw (11,1) ellipse (1 and 1)
      (11,3) ellipse (1 and 1);

\draw[style = dashed] (10.5,1) + (0:0.5 and 0.125) arc (0:180:0.5 and 0.125)
                      (11.5,1) + (0:0.5 and 0.125) arc (0:180:0.5 and 0.125)
                      (10.5,3) + (0:0.5 and 0.125) arc (0:180:0.5 and 0.125)
                      (11.5,3) + (0:0.5 and 0.125) arc (0:180:0.5 and 0.125);

\draw (10.5,1) + (0:0.5 and 0.125) arc (0:-180:0.5 and 0.125)
      (11.5,1) + (0:0.5 and 0.125) arc (0:-180:0.5 and 0.125)
      (10.5,3) + (0:0.5 and 0.125) arc (0:-180:0.5 and 0.125)
      (11.5,3) + (0:0.5 and 0.125) arc (0:-180:0.5 and 0.125);

\draw (11,0) -- (11,4);

\draw[->] (12.25,2) -- (13.25,2);

\draw (16,1) + (90:1.414 and 1.414) arc (90:180:1.414 and 1.414)
      (14,3) + (-90:1.414 and 1.414) arc (-90:0:1.414 and 1.414)
      (16,3) + (180:1.414 and 1.414) arc (180:270:1.414 and 1.414)
      (14,1) + (90:1.414 and 1.414) arc (90:0:1.414 and 1.414)

      (16,3) + (-90:0.5857 and 0.5857) arc (-90:180:0.5857 and 0.5857)
      (16,1) + (-180:0.5857 and 0.5857) arc (-180:90:0.5857 and 0.5857)
      (14,3) + (0:0.5857 and 0.5857) arc (0:270:0.5857 and 0.5857)
      (14,1) + (90:0.5857 and 0.5857) arc (90:360:0.5857 and 0.5857)

      (16,3) + (180:0.5857 and 0.146425) arc (180:360:0.5857 and 0.146425)
      (16,1) + (180:0.5857 and 0.146425) arc (180:360:0.5857 and 0.146425)
      (14,3) + (180:0.5857 and 0.146425) arc (180:360:0.5857 and 0.146425)
      (14,1) + (180:0.5857 and 0.146425) arc (180:360:0.5857 and 0.146425);

\draw[style = dashed]       (16,3) + (0:0.5857 and 0.146425) arc (0:180:0.5857 and 0.146425)
                            (16,1) + (0:0.5857 and 0.146425) arc (0:180:0.5857 and 0.146425)
                            (14,3) + (0:0.5857 and 0.146425) arc (0:180:0.5857 and 0.146425)
                            (14,1) + (0:0.5857 and 0.146425) arc (0:180:0.5857 and 0.146425);
\end{tikzpicture}
\caption{The smash product $(S^1 \vee S^1) \wedge (S^1 \vee S^1)$}
\label{fig:smash}
\end{figure}
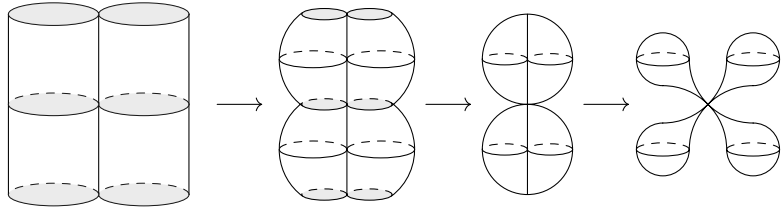

\begin{rmk}
Let $w = (1^{(m)},\epsilon^{(k)})$ and $w' = (1^{(m)},\epsilon^{(m+1)})$. If $m$ is prime and $\sigma \in S_m$ is a cycle of length $m$, then $\sigma$ acts on $\Delta_{0,w}$ by permuting the marks $\set{1,\dots,m}$. Then $\group \sigma$ is a free group action on $\Delta_{0,w}$. There is an inclusion $\Delta_{0,w} \hookrightarrow \Delta_{0,w'}$ that attaches the additional mark on top of the $k^\text{th}$ mark of weight $\epsilon$.

The quotient by the action of $\group \sigma$ on the colimit as $k \to \infty$ is a $K(\Z / m, 1)$ space. This process is analogous to Milnor's construction of $K(G,1)$ spaces in~\cite{milnor1956construction}.
\end{rmk}

\subsection{The permutation action}

Let $w = (1^{(m)},\epsilon^{(k)})$ and $0<\epsilon\le 1/k$. We reserve the heavy marks $(1)$ and $(2)$ for the construction of $X_{0,w}$. The action of permuting the remaining $k-2$ markings of weight $1$ and $k$ weights of weight $\epsilon$ gives rise to a representation on top dimensional homology.
The following theorem describes this homology representation.

\begin{thm} \label{representation theory}
Let $w = (1^{(m)}, \epsilon^{(k)})$ with $0<\epsilon\le 1/k$.
The action of $S_{m-2} \times S_k$ on $\Delta_{0,w}$ is a cellular map that permutes the $1$ marks and $\epsilon$ marks separately.
This pushes forward to an action on the top homology of $\Delta_{0,w}$.
Let $p_{m-1}(k)$ be the set of all partitions of $k$ into $m-1$ parts.
This action is the induced representation \[
\bigoplus_{\lambda \in p_{m-1}(k)} \operatorname{Ind}_{S_{\lambda_1} \times \dots \times S_{\lambda_{m-1}}}^{S_{m-2} \times S_k}(\mathrm{sgn}),
\]
where $\mathrm{sgn}$ is the alternating representation.
\end{thm}

\begin{proof}
Let $X_{0,w}$ be the heavy marking locus. The space $\Delta_{0,w}/ X_{0,w}$ consists of trees that are paths, labeled by the marks $\set{3,\ldots, m+k-2}$. The ordering of the marks in the path associates to each top dimensional cell of $\Delta_{0,w}/ X_{0,w}$ a permutation $\sigma \in S_{m+k-2}$. This map naturally extends to a map
\[
\phi: C_{m+k-4}\to V = \langle x_{\sigma} : \sigma \in S_{m+k-2} \rangle,
\]
where $C_{m+k-4}$ is the vector space of top dimensional chains. As a notational device, we write $x_{\sigma}$ in expanded form $x_{\sigma(1)} \otimes \cdots \otimes x_{\sigma(k-2)}$.

We consider the kernel of the boundary map as a subspace of $V$.
Each codimension one cell $\hat{x}$ of the CW structure fuses two $\epsilon$-marks $i$ and $j$ and thus is on the boundary of two cells, which we identify with the monomials $x_{ij} = \cdots \otimes x_i \otimes x_j \otimes \cdots$ and $x_{ji} = \cdots \otimes x_j \otimes x_i \otimes \cdots$.
If the boundary of a chain $\sum c_\alpha x_{\alpha}$ vanishes at $\hat{x}$, then the coefficients of $x_{ij}$ and $x_{ji}$ are negatives of one another, so $c_{ij} + c_{ji} = 0$.
More generally, fixing a placement of marks of weight $1$ and fixing a partition $\lambda$ of $\epsilon$-marks restricts to a $\term*{\prod_i \lambda_i !}$-dimensional vector space.
The restriction of the $S_{\lambda_1} \times \dots \times S_{\lambda_{m-1}}$ action on top homology is the alternating representation.
Moreover, for a fixed placement of heavy weights, the partitions of $\epsilon$-marks are in a natural one-to-one correspondence with the cosets of $S_{\lambda_1} \times \dots \times S_{\lambda_{m-1}}$ in $S_{k}$. The $S_{m-2}$ action on heavy weights commutes with the $S_k$ action on lightweights. Now $lambda$ determines the Each permutation of $m-2$ heavy marks is invariant under the $S_k$ action, so fixing this relative placement of heavy and light weights, and so the full $S_{m-2} \times S_k$ action is induced from the alternating representation.

Each partition $\lambda$ determines a $(S_{m-2} \times S_k)$-invariant subspace of homology, and so the full homology representation is simply the direct sum of these induced representations.
\end{proof}

We note that the group action in Theorem \ref{representation theory} fixes the two reserved marks of weight $1$. The action of the full permutation group remains unclear to the authors.

\begin{question}
Let $w = (1^{(m)},\epsilon^{(k)})$ for $0<\epsilon\le 1/k$. What is the representation of $S_m\times S_k$ on the vector space $H^{m+k-4}(\Delta_{0,w};\mathbf Q)$?
\end{question}

The answer is known when $k = 0$ via the main result of~\cite{robinson1996tree}, and when $m = 2$ where top homology is $1$-dimensional.

\section{Double covers, torsion, and disconnectedness}

This section offers a proof of Theorem \ref{B}. We begin with the observation that $\Delta_{0,w}$ consists entirely of paths if and only if $w = (w_1,\dots,w_n)$ is such that there is no partition $\set{1,\dots,n} = \alpha \sqcup \beta \sqcup \gamma$ such that $w(\alpha), w(\beta),w(\gamma) > 1$. Recall that we defined such spaces $\Delta_{0,w}$ as \textbf{path spaces}.

\begin{defn}
An \textbf{orientation} on a path is an injective assignment of $\set{\mathbf{L},\mathbf{R}}$ to the two leaves of the path.
For a path space $\Delta_{0,w}$, we denote by $\widetilde{\Delta}_{0,w}$ the moduli space of paths with orientation.
Let $\phi : \widetilde{\Delta}_{0,w} \to \Delta_{0,w}$ be the projection that forgets orientation.
\end{defn}

\begin{thm}\label{double cover}
Suppose $w = (w_1,\dots,w_n)$ is such that $\Delta_{0,w}$ is a path space.
Then $\widetilde{\Delta}_{0,w}$ is a natural double cover with projection map $\phi$.
Moreover, $\widetilde{\Delta}_{0,w}$ is isomorphic as a simplicial complex to the flag complex of the subposet $P$ of the boolean lattice $\mathcal{B}_n$ consisting of all $A \subset \set{1,\dots,n}$ such that $w(A) > 1$ and $w(A^{c}) > 1$.
\end{thm}

\begin{proof}
Observe that $\widetilde{\Delta}_{0,w}$ has the structure of a CW complex, arising in the same fashion as the CW structure on $\Delta_{0,w}$. If each tree in $\Delta_{0,w}$ is a path, it is clear that the forgetful map $\phi$ is continuous, surjective, and everywhere $2:1$. Moreover, $\phi$ is a cellular map, and hence an unramified double cover.

Each simplex of $\widetilde{\Delta}_{0,w}$ is uniquely determined by a label order, i.e.\ an ordered partition $(D_1, \dots, D_r)$ of the set of marks $\set{1,\dots,n}$ as in Figure \ref{label order fig}.
This defines a map $f$ from simplices to ordered partitions.
Conversely, an ordered partition corresponds to a simplex if and only if $w(D_1) > 1$ and $w(D_r) > 1$.

\begin{figure}
\begin{tikzpicture}[line cap = round, line join = round]
\coordinate (pm4) at (-3.5, 0);
\coordinate (pm3) at (-2.5, 0);
\coordinate (pm2) at (-1.5,0);
\coordinate (pm1) at (-0.5,0);
\coordinate (pp1) at (0.5,0);
\coordinate (pp2) at (1.5,0);
\coordinate (pp3) at (2.5, 0);
\coordinate (pp4) at (3.5, 0);

\node (l4) at (-3.5, 1.3) {$D_1$};
\node (l3) at (-2.5, 1.3) {$D_2$};
\node (l2) at (-1.5, 1.3) {$D_3$};
\node (l5) at (1.5, 1.3) {$D_{r-2}$};
\node (l6) at (2.5, 1.3) {$D_{r-1}$};
\node (l7) at (3.5, 1.3) {$D_r$};

\draw[darkgray,dashed] (pm4) -- ++(90:1)
            (pm4) -- ++(80:1)
            (pm4) -- ++(100:1)
            (pm3) -- ++(90:1)
            (pm3) -- ++(80:1)
            (pm3) -- ++(100:1)
            (pm2) -- ++(90:1)
            (pm2) -- ++(80:1)
            (pm2) -- ++(100:1)
            (pp2) -- ++(90:1)
            (pp2) -- ++(80:1)
            (pp2) -- ++(100:1)
            (pp3) -- ++(90:1)
            (pp3) -- ++(80:1)
            (pp3) -- ++(100:1)
            (pp4) -- ++(90:1)
            (pp4) -- ++(80:1)
            (pp4) -- ++(100:1);

\draw[darkgray,dotted] (pm1)--(pp1);
\draw (pm1) -- (pm2) -- (pm3) -- (pm4)
      (pp1) -- (pp2) -- (pp3) -- (pp4);

\node (L) at (-3.5,-0.5) {$\color{Firebrick4}{\mathbf{L}}$};
\node (R) at (3.5,-0.5) {$\color{DodgerBlue4}{\mathbf{R}}$};
\end{tikzpicture}
\caption{An ordered partition}
\label{label order fig}
\end{figure}
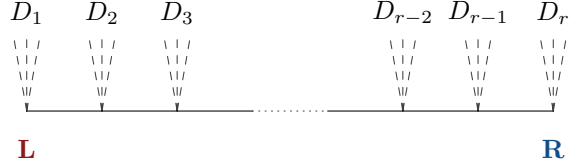

Let $P$ be the subposet of the boolean lattice $\mathcal{B}_n$ consisting of all $A \subset \set{1,\dots,n}$ for which $w(A) > 1$ and $w(A^{c}) > 1$.
Denote by $F$ the map that sends each simplex $\sigma$ to the chain $\set{ \bigcup_{i=1}^j D_i \mid 1 \le j \le r-1 }$, where $f(\sigma) = (D_1, \dots, D_r)$.
This sends simplices into flags of $P$.
The flags $\set{ E_i : 1 \le i \le r-1 }$ that appear are precisely those for which $w(E_i) > 1$ and $w(E_i^c) > 1$ for all $i$.
Given a flag $E_1 \subset \cdots \subset E_r$ of $P$, the inverse map $G$ sends $\set{ E_1, \dots, E_r }$ to the partition $(E_1, E_2 \setminus E_1, E_3 \setminus E_2, \dots, E_r \setminus E_{r-1}, \set{1,\dots,n} \setminus E_r)$.
Since $w(E_1) > 1$ and $w(\set{1,\dots,n} \setminus E_r) > 1$, there is a unique simplex corresponding to the resulting partition.
This defines a bijection between faces.

Given a simplex $\sigma$ in $\widetilde{\Delta}_{0,w}$, a face of $\sigma$ maps to the ordered partition $(D_1, \dots, D_{i-1}, D_i \cup D_{i+1}, D_{i+2}, \dots, D_r)$.
This, in turn, maps to the flag $\set{ E_1, \dots, \widehat{E}_i, \dots, E_{r-1} }$ under $F$, which excludes the index $i$.
Similarly, flags $\set{ E_1, \dots, \widehat{E}_i, \dots, E_{r-1} }$ map to simplices corresponding to the order partition $(D_1, \dots, D_{i-1}, D_i \cup D_{i+1}, D_{i+1}, \dots, D_r)$.
Therefore, since the boundary relation is preserved, this bijection is indeed an isomorphism.
\end{proof}

We can specialize this result to obtain a two-parameter infinite family of spaces $\Delta_{0,w}$ with torsion in its fundamental group.

\begin{defn}
The \textbf{rank selected poset} of the Boolean lattice $\mathcal{B}_{n,k}$ consists of all sets $A$ such that $k \le \abs{A} \le n-k$, under inclusion.
\end{defn}

It is known that $\mathcal{B}_{n,k}$ is shellable and therefore has the homotopy type of a wedge of spheres.
This is implicit in \cite{lovasz1978kneser}, and an explicit statement can be found in \cite[Corollary C.10]{de2012course}.
$\mathcal{B}_{n,k}$ is also the face complex associated to the neighborhood complex of the Kneser graph \cite{lovasz1978kneser}.

\begin{cor}
Let $w = (\frac{1}{k}^{2(k+1)+m})$, where $1 \le m \le k$. Then, $\pi_1(\Delta_{0,w}) \cong \Z / 2$.
\end{cor}
\begin{proof}
A set of markings $A \subset \set{1,\dots,2(k+1) + m}$ has $w(A) > 1$ and $w(A^c) > 1$ if and only if $k < \abs A < 2(k+1) + m - k$, so Theorem \ref{double cover} implies that $\widetilde{\Delta}_{0,w}$ is isomorphic as a simplicial complex to the flag complex of the shellable poset $\mathcal{B}_{2(k+1)+m,\,k+1}$. Consequently, $\widetilde{\Delta}_{0,w}$ has the homotopy type of a wedge of spheres of dimension $m$. If $m \ge 2$, then $\widetilde{\Delta}_{0,w}$ is simply connected and hence the universal cover of $\Delta_{0,w}$. We conclude that $\pi_1(\Delta_{0,w}) = \Z / 2$.
\end{proof}

\subsection{A disconnected family}
To finish the proof of Theorem \ref{B}, we identify an infinite family of weight vectors $w$ for which $\Delta_{0,w}$ is disconnected.

\begin{prop}
Let $w = (1/m^{(2m)}, \epsilon^{(k)})$ with $0 < k \cdot \epsilon < 1/m$ and $k \ge 2$.
Then $\Delta_{0,w} \cong \bigsqcup_{\frac{1}{2}\binom{2n}{n}} S^{k-2}$.
\end{prop}

\begin{proof}
The stable trees in this space must have $m$ marks of weight $1/m$ on each leaf.
There are $\frac{1}{2}\binom{2m}{m}$ ways of performing this partition. If $k \ge 3$, each partition produces a path-connected component with the specified partition, with the partition on each leaf and at least $1$ mark of weight $\epsilon$ on each leaf.
Each component is homeomorphic to $\Delta_{0, w'}$, where $w' = (1, 1, \epsilon^{(k)})$. This is a $(k-2)$-sphere by Theorem \ref{heavy light}.

We handle the case $k = 2$ separately; the stable trees must have $m$ marks of weight $1/m$ on each leaf and one $\epsilon$ weight on each leaf.
There are $\frac{1}{2} \binom{2m}{m}$ choices of this partition of $m$ marks and $2$ choices for the $\epsilon$-weight, so this case is a set of $\binom{2m}{m}$ isolated points, which is the disjoint union of the correct number of $0$-spheres.
\end{proof}

\section{Topology of $\Delta_{1,w}$}\label{sec: genus1}

When all weights are equal to $1$, the subcomplex of curves with repeated markings in $\Delta_{g,n}$ is a large contractible subcomplex whose complement has a simple combinatorial description. This is proved and used extensively by Chan, Galatius, and Payne in \cite[Section 9]{CGP}. For general weights, this role is played by the locus of curves with heavy markings.

\begin{defn}
Let $g>0$. The \textbf{heavy marking locus} of $\Delta_{g,w}$, denoted $X_{g,w}$, is the subspace of $w$-stable tropical curves $G \in\Delta_{g,w}$ for which there is a vertex $v$ such that $w(m_G^{-1}(v))>1$.
\end{defn}

\noindent The analogue of their argument for the contractibility of the repeated marking locus in $\Delta_{g,n}$ establishes the following lemma in the general $\Delta_{g,w}$ case.

\begin{lemma}\label{heavy marking locus}
Let ${{w}}$ be a weight vector such that
\[
\sum_{i=1}^n {{w}}_i >1,
\]
and suppose $g>0$. Then, $X_{g,w}$ is a contractible subcomplex of $\Delta_{g,{{w}}}$.
\end{lemma}
\begin{proof}
We refer the reader to \cite[Section 8]{CGP} for the structure of the proof, and merely remark here on the necessary changes for its generalization to the weighted case.

Let $P$ denote the point corresponding to the connected tree with exactly two vertices, one of genus $g$ and the other marked by $\set{1,2,\dots,n}$. In loc.~cit.~the authors exhibit a series of strong deformation retractions which contract $X_{g,w}$ onto $P$. To generalize their argument to the case of an arbitrary weight vector, one need first guarantee that $P \in X_{g,w}$ as each retraction increases the $w$-weighted valency at a vertex. Since $P \in X_{g,w}$ if and only if $\sum_{i=1}^n w_i>1$, we conclude that $P \in X_{g,w}$ by our hypothesis.

We also require that their series of strong deformation retractions never destabilizes a $w$-weighted tropical curve. This can only occur when a retraction grows a distinguished bridge at a core vertex, removing marks from a core vertex and placing them onto a leaf vertex. So long as the total weight at that vertex is more than one, this does not destabilize the curve. Consequently, their argument generalizes here and $X_{g,w}$ is contractible. \edit{P13Q3: no changes were made to this proof.}
\end{proof}

We can characterize the set complement $\Delta_{1,w}\setminus X_{1,w}$ in a way similar to Lemma \ref{genus 0 post contraction}.
Let $\T^r$ denote the compact $r$-torus $(S^1)^r$. Throughout this section, we identify $\T^r$ with $(S^1)^r \subset \C^r$. Unless otherwise specified, $S^1 \subset \C$ acts on $\T^r$ by coordinatewise scalar multiplication, and $\Z/2$ acts on $\T^r$ by coordinatewise complex conjugation.

\begin{lemma}\label{genus 1 post contraction}
Let $w = (w_1,\dots,w_n)$ and $\mathcal A = \set{A \in 2^{\set{1,\dots,n}} \mid \sum_{i \in A} w_i > 1}$. Denote by $\T^{n-1} \subset \T^n$ the set of tuples where the first coordinate $x_1 = 1$ is fixed. Let
\[ D^n = \set{ x \in \T^n \mid x_{i_1} = \dots = x_{i_k} \text{ for some }
\set{i_1,\dots,i_k} = A \in \mathcal A} \] and let $D^{n-1} = \T^{n-1} \cap D^n$. Then we have the homeomorphisms
\[
\Delta_{1, w} \setminus  X_{1,w} \cong \frac{\T^n \setminus D^n}{S^1 \rtimes \Z/2} \cong \frac{\T^{n-1} \setminus D^{n-1}}{\Z/2}.
\]
\end{lemma}

\begin{proof}
A tropical curve $G \in \Delta_{1,w} \setminus X_{1,w}$ must be its own core: it can have no leaves, since any leaf will be marked with weight $> 1$.
Since the genus of $G$ equals $1$, the graph $G$ is a cycle with total length $1$. There is therefore a continuous surjection
\[
\T^n\setminus D^n\to \Delta_{1,w}\setminus X_{1,w}
\]
sending a tuple $(e^{2\pi i \, x_1},\dots,e^{2\pi i \, x_n}) \in \T^n$ to the metric marked cycle with mark $(1)$ at some vertex and $(i)$ at distance $x_i-x_1$ from $(1)$ in a specified direction.

Fix $G \in \Delta_{1,w} \setminus X_{1,w}$. The fiber over $G$ consists of all rotations and reflections of $(1,e^{2\pi i \, x_2},\dots, e^{2\pi i \, x_n})$ in $\T^n \setminus D^n$, where $x_i$ is the distance from
$(1)$ to $(i)$ along $G$ in a specified direction. This fiber is precisely the orbit $(S^1 \rtimes \Z/2) \cdot (1,e^{2\pi i \, x_2},\dots, e^{2\pi i \, x_n})$.

The identification of the fiber over $G$ with an orbit of $S^1 \rtimes \Z/2$ yields the leftmost homeomorphism. We choose a representative for the $S^1$ action on $\T^n \setminus D^n$ by setting the first coordinate equal to $1\in S^1$: this identifies $(\T^n \setminus D^n)/S^1 \cong \T^{n-1} \setminus D^{n-1}$, yielding the rightmost homeomorphism.
\end{proof}

When $w=(1,\epsilon)$, we identify the complement
\[
\Delta_{1,w} \setminus X_{1,w} \cong ((\T^2 \setminus D^2)/S^1)/\Z/2 \cong (S^1 \setminus \set{1}) / \Z/2.
\]
This $\Z/2$ action is by complex conjugation, and so topologically $\Delta_{1,w} \setminus X_{1,w}$ is homeomorphic to the half-open interval $(0,1/2]$ where $1/2$ is the image of the fixed point under the action. This behaviour is representative of a general phenomenon. In cases of such exceptional weight vectors, we are sometimes forced to work with quotients of disks by finite groups.
If $w=(1,w_2,\dots,w_n)$ where $\sum_{i=2}^n w_i >1$, this issue disappears because elements of $\Delta_{1,w} \setminus X_{1,w}$ will have at least three vertices. A tropical curve $G \in \Delta_{1,w} \setminus X_{1,w}$ is a fixed point of the $\Z/2$ action if and only if all labels $\set{1,\dots, n}$ mark the points $\set{\pm 1} \subset S^1$: that is, if $G$ has only two vertices. But all such graphs lie in $X_{1,w}$. As a consequence, the tropical curves parameterized by $\Delta_{1,w}\setminus X_{1,w}$ have only trivial automorphisms.

\begin{rmk}\label{rmk:genus 1 pt compactification}
If $\Delta_{1,w}$ is connected and $X_{1,w}$ is nonempty, then $\Delta_{1,w}$ is homotopy equivalent to the one point compactification of $\Delta_{1,w}\setminus X_{1,w}$.
\end{rmk}
To determine the topology of $\Delta_{1,w}$ we require the following lemma.

\begin{lemma}\label{subspace identification}
Let $w=(1,1,w_3,\dots,w_n)$ and $D$ as in Section \ref{sec:structural}. Then, $\Delta_{0,w} \setminus X_{0,w} \times \R^2 \cong \R^{n-2} \setminus D$.
\end{lemma}
\begin{proof}
We have immediately from the remarks following Proposition \ref{prop:global path space} and the homeomorphism $\R_{>0} \cong \R$ that $\Delta_{0,w} \setminus X_{0,w} \times \R^2 \cong (S \setminus D) \times \R_{> 0} \times \R$. Now define \[ \varphi: (S \setminus D) \times \R_{> 0} \times \R \rightarrow \R^{n-2} \] by $\varphi(x,s,t)=t \cdot x + s \cdot \mathbf{1}$.
Since $D$ is nonempty and invariant under translation by multiples of $\mathbf{1}$ and scalar multiplication, this defines a homeomorphism onto the image $\R^{n-2} \setminus D$.
\end{proof}

\begin{thm}\label{genus one action}
Let ${{w}} = (1, w_3,\dots, w_n)$ and ${{w}}' = (1, 1, w_3,\dots,w_n)$.
Suppose further that $\sum_{i=3}^n w_i > 1$ and
let $\Z/2$ act on $\Delta_{0,w'}$ by transposing the two distinguished marks of weight $1$. Then \[
\Susp^2(\Delta_{0,{{w}}'}) / (\Z/2) \simeq \Delta_{1,w}.
\]
\end{thm}

\begin{proof}
By Remark \ref{rmk:genus 1 pt compactification}, $\Delta_{1, w}/X_{1,w}$ is homeomorphic to the one point compactification of $(\T^{n-2} \setminus D^{n-2})/(\Z/2)$, where the $\Z/2$ action is by componentwise complex conjugation. We observe that this quotient is also homeomorphic to $(\T^{n-2} / D^{n-2})/(\Z/2)$, where $\Z/2$ acts by componentwise complex conjugation on $\T^{n-2} \setminus D^{n-2}$ and as identity on the point $D^{n-2}/D^{n-2}$.

Since a graph parametrized by $\T^{n-2} \setminus D^{n-2}$ has the heavy mark $(1)$ at $1 \in S^1$, the other coordinates cannot equal $1$.  Consequently, $\T^{n-2} \setminus D^{n-2} \subset (S^1 \setminus \set{1})^{n-2} \cong \R^{n-2}$, where the homeomorphism is componentwise stereographic projection. In fact, the image of $D^{n-2}$ under this homeomorphism to $\R^{n-2}$ is precisely the diagonal subspace arrangement $D$ of Section \ref{sec:structural} for $\Delta_{0,w'}$. Invoking Lemmas~\ref{genus 1 post contraction} and~\ref{subspace identification},
\[
\T^{n-2} \setminus D^{n-2} \cong \R^{n-2} \setminus D \cong \Delta_{0, w'} \setminus X_{0,w'} \times \R^2.
\]
We identify the one point compactification of $\Delta_{0,w'} \setminus X_{0,w'} \times \R^2$ with the smash product $\Delta_{0,w'} / X_{0,w'} \wedge S^2 \simeq \Susp^2(\Delta_{0,w'}/X_{0,w'})$, see~\cite[pg.\,199]{bredontopology}. This series of homeomorphisms induces a $\Z/2$ action on $\Susp^2(\Delta_{0,w'} / X_{0,w'})$. This gives the homotopy equivalence
\[ \Delta_{1,w} \simeq (\T^{n-2} / D^{n-2})/(\Z/2) \cong \Susp^2(\Delta_{0,w'} / X_{0,w'})/(\Z/2) \simeq \Susp^2(\Delta_{0,w'})/(\Z/2). \]
\end{proof}

Theorem \ref{C} is now a consequence of Corollary \ref{genus one 1,1,alpha} and Theorem \ref{genus one heavy light} below.

\begin{cor}\label{genus one 1,1,alpha}
Let $w = (1,1,w_4,\dots,w_n)$ and $w' = (1,1,1,w_4,\dots,w_n)$ with $n \ge 4$.
Then
\[ \Susp^2(\Delta_{0,w'}) \simeq \Delta_{1,w} \vee \Delta_{1,w}.
\]
In particular,
$\Delta_{1,w}$ is homotopic to a wedge of spheres with half as many spheres of each dimension as $\Susp^2(\Delta_{0,w'})$.
\end{cor}
\begin{proof}
By the proof of Theorem \ref{genus one action}, we have a homeomorphism
\[
\Delta_{1, w} \setminus X_{1,w}\cong (\R^{n-2} \setminus D) / (\Z/2).
\]
If $w_3=1$, then $\Z/2$ acts freely on $\R^{n-2} \setminus D$. The half-space $\set{x_1 < x_2} \cap (\R^{n-2} \setminus D) \subset \R^{n-2}$ forms a fundamental domain for this action. $\Delta_{1,w}$ is the one point compactification of this fundamental domain, while $\Susp^2(\Delta_{0,w'})$ is the one point compactification of $\R^n \setminus D$. Hence \[
\Susp^2( \Delta_{0,w'} ) \simeq \Delta_{1,w} \vee \Delta_{1,w}.
\]
\end{proof}
\subsection{The heavy/light weight data}

We specialize to the case $w=(1^{(m)},\epsilon^{(k)})$, where $m+k \ge 1$ and $k \cdot \epsilon <1$. Utilizing Theorem \ref{genus one action}, we give the following classification of the Betti numbers of $\Delta_{1,w}$.

\begin{thm}\label{genus one heavy light}
Let $w=(1^{(m)},\epsilon^{(k)})$ for $m+k \ge 1$ and $k \cdot \epsilon <1$. Let $\beta_d$ denote the $d^\text{th}$ Betti number, computed with reduced rational homology. We have the following cases:
\begin{enumerate}
\item If $m \ge 2$ and $k \ge 1$, then \[ \beta_d(\Delta_{1,w}) = \frac12 (m-1)!\,m^k  \qquad \text{if}~ d = m + k - 1, \]
and $0$ otherwise.
\item If $m=1$ and $k \ge 1$, then \[ \beta_d(\Delta_{1,w}) = 1 \qquad \text{if}~d=k~\text{and}~k~\text{is even}, \]
and $0$ otherwise.
\item If $m=0$ and $k \ge 1$, then \[ \beta_d(\Delta_{1,w}) =
\binom{k-1}{d} \qquad \text{if}~ 0 < d < m ~\text{and}~ d ~\text{is even}, \]
and $0$ otherwise.

\end{enumerate}
\end{thm}

Theorem \ref{genus one heavy light} is established using Propositions \ref{m=2 heavy light}\,--\,\ref{no 1 heavy light}. The $k = 0$ case is handled by~\cite[Theorem 1.2]{CGP}. We begin with the generic case, which follows quickly from our work above.

\begin{prop}\label{m=2 heavy light}
Let $w = (1^{(m)}, \epsilon^{(k)})$ for $k \cdot \epsilon<1$, where $m \ge 2$ and $m + k \ge 3$.
Then $\Delta_{1,w}$ is a wedge of $\frac12 (m-1)! \,m^k$ spheres of dimension $m+k-1$.
\end{prop}

\begin{proof}
The result is immediate from Corollary \ref{genus one 1,1,alpha} and Theorem \ref{heavy light}.
\end{proof}
When $w = (1^{(m)}, \epsilon^{(k)})$, $k \cdot \epsilon <1$, and $m \le 1$, we observe that $\Delta_{1,w}$ behaves a bit differently than the generic case. To finish the computation of the Betti numbers, we use the following fact from Smith theory \cite[III, 2.4]{bredon1972introduction}:
\begin{prop}\label{smith-invariant-classes}
Let $G$ be a finite group acting simplicially on $K$ and $\Lambda$ a field of characteristic $0$ or prime to $\abs G$.
We have the natural isomorphism:
\[
H_i(K / G ; \Lambda) \cong H_i(K; \Lambda)^{G}.
\]
\end{prop}

\begin{prop}
Let $w = (1, \epsilon^{(k)})$ with $k \cdot \epsilon <1$ and $k \ge 1$. Then $\beta_d(\Delta_{1,w}) = 1$ if $d=k$ and $k$  is even and $0$ otherwise.
\end{prop}

\begin{proof}
For the given weight vector, the only forbidden collisions are with the first marked point. It follows that $\T^{k} \setminus D^{k}$ is homeomorphic to $\R^k$ by stereographic projection on each coordinate. The action of $\Z / 2$ is coordinatewise negation. Now, $\Delta_{1,w}$ is homotopic to the quotient of the one point compactification of $\R^k$ by this action.
This one point compactification is a $k$-sphere. The $\Z/2$ action preserves a nontrivial volume form on $S^k$ if and only if $k$ is even, so the result follows by an application of Proposition~\ref{smith-invariant-classes} above.
\end{proof}

\begin{prop}\label{no 1 heavy light}
Let $w = (\epsilon^{(k)})$, where $k \cdot \epsilon <1$. The Betti numbers computed with rational homology are
\[ \beta_d(\Delta_{1,w}) =
\binom{k-1}{d} \qquad \text{if}~ 0 < d < m ~\text{and}~ d ~\text{is even}, \]
and $0$ otherwise.
\end{prop}

\begin{proof}
The heavy marking locus is empty, so the proof of Lemma \ref{genus 1 post contraction} shows that $\Delta_{1, w} \cong \T^{k-1} / (\Z / 2)$.
Recall that the $\Z / 2$ action on $\T^{k-1}$ is by complex conjugation of each coordinate.
The result then follows from the fact that $H_i(\T^{k-1} / (\Z/2) ; \Q) \cong H_i(\T^{k-1} ; \Q)^{\Z/2}$,
noting that the cohomology ring of $\T^{k-1}$ is the exterior algebra $\Lambda_{k}[x_1, \dots, x_{k-1}]$, and that each odd dimensional cycle changes sign under the $\Z / 2$ action.
\end{proof}

\bibliographystyle{siam}
\bibliography{toptrop}
\bigskip

{\small {\sc Yale University, New Haven, CT 06520} \par
\textit{Email address, Alois Cerbu}: {\tt \href{mailto:alois.cerbu@yale.edu}{alois.cerbu@yale.edu}} \par
\textit{Email address, Luke Peilen}: {\tt \href{mailto:luke.peilen@yale.edu}{luke.peilen@yale.edu}} \par
\textit{Email address, Andrew Salmon}: {\tt \href{mailto:andrew.salmon@yale.edu}{andrew.salmon@yale.edu}} \\

{\sc The College of New Jersey, Ewing, NJ 08628} \par
\textit{Email address, Steffen Marcus}: {\tt \href{mailto:marcuss@tcnj.edu}{marcuss@tcnj.edu}} \\

{\sc Massachusetts Institute of Technology, Cambridge, MA 02139}\par
\textit{Email address, Dhruv Ranganathan}: {\tt \href{mailto:dhruvr@mit.edu}{dhruvr@mit.edu}}

\end{document}